\documentclass[a4paper,11pt,reqno]{amsart}
\usepackage[utf8]{inputenc}
\usepackage[T1]{fontenc}
\usepackage[english]{babel}
\usepackage[
  left=2cm,
  right=2cm,
  top=2.5cm,
  bottom=2.5cm
]{geometry}
\usepackage{mathtools}
\usepackage{amssymb}
\usepackage{mathrsfs}
\IfFileExists{newtxtext.sty}{\usepackage{newtxtext}}{}
\usepackage{microtype}
\usepackage{xcolor}

\usepackage{enumitem}
\usepackage{array,booktabs}

\usepackage{tcolorbox}
\allowdisplaybreaks %sirve para permitir que LaTeX divida una ecuación de varias líneas entre dos páginas.

\numberwithin{equation}{section}

\makeatletter

\long\def\guardarfrag#1#2{%
  \expandafter\long\expandafter\def
  \csname frag@#1\endcsname{#2}%
}

\newcommand{\fragmento}[1]{%
  \@ifundefined{frag@#1}
    {%
      \PackageError{fragmentos}
        {Fragment `#1' is not defined}
        {Check the fragment name.}%
    }
    {%
      \csname frag@#1\endcsname
    }%
}

\makeatother

\newtheorem{proposition}{Proposition}[section]
\newtheorem{theorem}[proposition]{Theorem}
\newtheorem{lemma}[proposition]{Lemma}

\theoremstyle{definition}

\newtheorem{definition}[proposition]{Definition}

\newtheorem{remark}{Remark}

\newcommand{\din}{\textnormal{d}}
\newcommand{\funct}{\mathtt{L}}
\newcommand{\loc}{\textnormal{loc}}
\newcommand{\R}{\mathbb{R}}
\newcommand{\qpu}{{\bf q}}
\newcommand{\lbp}{{\bf c}}
\newcommand{\fnu}{\mathtt{F}}
\newcommand{\ftes}{\varphi}

\usepackage{etoolbox}

\makeatletter

\patchcmd{\@tocline}
  {\hfil}
  {\leaders\hbox to 0.6em{\hss.\hss}\hfill}
  {}{}

\renewcommand{\@pnumwidth}{1.5em}

\def\l@section{\@tocline{1}{0.6em}{0em}{}{}}         % Sección sin sangría
\def\l@subsection{\@tocline{2}{0.3em}{2em}{}{}}      % Subsección con 2em
\def\l@subsubsection{\@tocline{3}{0.2em}{4.7em}{}{}}   % Subsubsección con 4em

\patchcmd{\@tocline}
  {\vskip #1}
  {\vskip #1\relax}
  {}{}

\makeatother

\title{Asymptotic Behavior of Radial Solutions to Singular $p$-Laplacian Equations as $p\to1$}

\author{Juan Pablo Alcon Apaza}

\address{Juan Pablo Alcon Apaza, Departamento de Matem\'atica, Universidade Federal de Minas Gerais, 31270-901, Belo Horizonte - MG,  Brazil}

\email{juanpabloalconapaza@gmail.com}

\date{}

\begin{document}

\begin{abstract}
We study the singular limit $p\downarrow1$ for positive radial entire solutions $u_p$ of
\[
-\Delta_pu_p=F_p(|x|,u_p,|\nabla u_p|)u_p^{-\beta_p}
\quad\text{in }\mathbb R^n,
\]
where $n\ge2$, $1<p<\min\{2,\sqrt n\}$, and $0\le\beta_p\le p-1$. Assuming that
\[
F_p\to F_1\quad\text{locally uniformly},
\qquad
\frac{\beta_p}{p-1}\to\lbp\in[0,1]
\quad\text{as }p\downarrow1,
\]
we prove that, after passing to a subsequence, $u_{p_j}\to u$ in $C_{\loc}(\mathbb R^n)$ and
 $\nabla u_{p_j}\stackrel{*}{\rightharpoonup}\nabla u $
in $L^\infty_{\loc}(\mathbb R^n;\mathbb R^n)$, where $u\in W^{1,\infty}(\mathbb R^n)$ is radial, and  solves
\[
-\Delta_1u=\mathtt K(|x|)
\quad\text{in }\mathbb R^n.
\]
On $\{x\in\mathbb R^n\mid\mathtt G(|x|)<1\}$, we prove that $\mathtt K(|x|)=w^{-\lbp}F_1(|x|,u(x),0)$ and that $\nabla u=0$ a.e., where $u_{p_i} ^{p_i-1} \to w$. On the other hand, we show that if $\rho\mapsto F_1(r,t,\rho)$ is affine for every $r$ and $t$, then $\mathtt K(|x|)=w^{-\lbp}F_1(|x|,u(x),|\nabla u(x)|)$
 a.e. in $\mathbb R^n$.

For $n=1$ and $1<p<2$, we study
\[
\frac{\din}{\din x}\bigl(|u_p'|^{p-2}u_p'\bigr)
=\mathtt F_p(|x|,u_p,|u_p'|)u_p^{-\beta_p}
\quad\text{in }\mathbb R,
\]
and obtain analogous results as $p\downarrow1$.
\end{abstract}

\maketitle

\tableofcontents

\let\thefootnote\relax
\let\thefootnote\relax

%%
%%
%%
%%
%%
%%
%%%
%%

\section{Introduction}

We study positive radial entire solutions of
\begin{equation}\label{31}
-\Delta_p u=F_p(|x|,u,|\nabla u|)u^{-\beta_p}
\quad\text{in }\mathbb R^n,
\end{equation}
where $n\ge2$, $1<p<\min\{2,\sqrt n\}$, $0\le\beta_p<p-1$, and $F_p$ satisfies assumptions \ref{1}--\ref{7} on page \pageref{1}.

Following the technique used in \cite[Proposition 3.1]{MiaoYang2008}, for each $p\in(1,\min\{2,\sqrt n\})$ we obtain a solution $u_p(x)=v_p(|x|)$ of \eqref{31} by applying the Schauder--Tychonoff theorem to an integral operator $T:Q\to C^2 ([0,\infty))$. The resulting function $v_p\in Q$ satisfies
\[
\xi_1^{-1}\ftes_p(r)\le v_p(r)\le\xi_1,
\qquad |v_p'(r)|\le\xi_2,
\]
where $\ftes_p(r):=\max\{\ell,r\}^{-\frac{n-p}{p-1}}$. These bounds on $v_p$ and $v_p'$ are the basis of our analysis as $p\downarrow1$, {with $\ell$, $\xi_1$, and $\xi_2$ independent of $p$}. The work of Mercaldo, Segura de Le\'on, and Trombetti \cite{MercaldoSeguraTrombetti2008} is a main motivation for this limit analysis.

For $n=1$, {we adapt to one dimension the construction of radial solutions used by Qi \cite{Qi2010} for $p\ge n\ge2$}. We obtain even solutions $u_p(x)=y_p(|x|)$ of
\begin{equation}\label{117}
\frac{\din}{\din x}\left(|u_p'(x)|^{p-2}u_p'(x)\right)
=\fnu_p\left(|x|,u_p(x),|u_p'(x)|\right)u_p(x)^{-\beta_p},
\qquad x\in\mathbb R,
\end{equation}
where $\beta_p\ge0$ and $1<p<2$. Under the corresponding assumptions on $\fnu_p$, these solutions satisfy
\begin{gather*}
1\le y_p(r)\le\xi_1(1+\ftes(r)),\qquad 0\le y_p'(r)\le\xi_2,\\
y_p(0)=1,\qquad y_p'(0)=0,
\end{gather*}
where $\ftes(r):=\max\{\ell,r\}$. These estimates also allow us to study the limit $p\downarrow1$ { when the constants are independent of $p$ and the one-dimensional limit assumptions hold}.

\subsection{Context and related work}

For fixed $p$, Gon\c{c}alves and Santos \cite{GoncalvesSantos2004} considered
\begin{equation*}
\left\{\begin{aligned}
-\Delta_p u&=\rho(x)f(u)&&\text{in }\mathbb R^n,\\
u&>0&&\text{in }\mathbb R^n,\\
u(x)&\to0&&\text{as }|x|\to\infty,
\end{aligned}\right.
\end{equation*}
with $f:(0,\infty)\to(0,\infty)$ possibly singular at $0$ and $\rho:\mathbb R^n\to[0,\infty)$.  For $n\ge3$, radial $\rho$, and $1<p<n$, their Theorem 1.1 gives a positive radial solution. The proof uses fixed point arguments, shooting, and subsolutions and supersolutions.

For $\lambda>0$, $p>1$, and a smooth bounded domain $\Omega\subset\mathbb R^n$, $n\ge2$, positive solutions of
\[
\left\{\begin{aligned}
-\Delta_p u&=\lambda f(u)&&\text{in }\Omega,\\
u&=0&&\text{on }\partial\Omega
\end{aligned}\right.
\]
have been studied under various assumptions on $f$ and $\Omega$, see \cite{zbMATH01545349, zbMATH04186196, zbMATH00147917} and the references therein.

When $f:(0, \infty) \rightarrow(0, \infty)$ and $q: \mathbb{R}^n \rightarrow(0, \infty)$ are continuous functions, and
\begin{equation*}
\int_1^{\infty}\left(\int_0^t f(s) \, \din s\right)^{-1 / p} \, \din t=\infty 
\end{equation*}
it has been shown in \cite{zbMATH01112213} that there {exist} entire radially symmetric solutions of the problem
\begin{equation*}
\Delta _p u=q{(x)} f(u), \quad \text { in } \mathbb{R}^n .
\end{equation*}

For equations with gradient dependence, Yang \cite{zbMATH05164262} provided sufficient conditions for
\begin{equation*}
\Delta_p u=f(x,u,\nabla u)\quad\text{in }\mathbb R^n
\end{equation*}
to admit infinitely many positive entire solutions, each bounded above and bounded away from zero. Wu  \cite{zbMATH05146781} studied bounded positive entire solutions of
\begin{equation*}
-\Delta u=f(x,u,\nabla u)u^{-\beta}
\quad\text{in }\mathbb R^n,\qquad n\ge3,
\end{equation*}
where $\beta\in(0,1)$.

Miao and Yang \cite{MiaoYang2008} { extended part of the results in} \cite{zbMATH05164262, zbMATH05146781}, and considered the problem
\begin{equation*}\label{eq:intro-miao-yang}
-\Delta _p u =f(x,u,\nabla u)u^{-\beta}
\quad\text {in }\mathbb R^n,
\end{equation*}
{ where} $1<p\le n$ { and} $0\le\beta<p-1$, and obtained { a} decaying positive entire { solution} by the sub--supersolution method.

For $p\ge n\ge2$ and $\beta\ge0$, Qi \cite{Qi2010} used the subsolution--supersolution method to establish the existence of positive entire solutions of
\begin{equation*}
\Delta_p u= \textnormal{f}(x,u,\nabla u)u^{-\beta}
\quad\text{in }\mathbb R^n.
\end{equation*}

Other results involving gradient dependence include the following. Guarnotta, Marano, and Moussaoui \cite{GuarnottaMaranoMoussaoui2022} established the existence of positive entire solutions for singular quasilinear convective systems in $\mathbb R^n$, while Gambera and Guarnotta \cite{GamberaGuarnotta2022} studied a strongly singular convective equation driven by a nonhomogeneous operator modeled on the $(p,q)$-Laplacian. For a survey, see \cite{GuarnottaLivreaMarano2022}.

\medskip

\noindent\textit{The limit problem $p\downarrow1$.} Among the early contributions to the theory of the $1$-Laplacian are the works of Andreu, Ballester, Caselles, and Maz\'on on the total variation flow \cite{Andreu2001DirFl, Andr2000MinFl, Andr2001MinVa}, together with the monograph by Andreu-Vaillo, Caselles, and Maz\'on \cite{Andr2004par}. In \cite{Andreu2001DirFl}, the authors used Anzellotti's pairing theory \cite{Anzellotti1983} to give a weak interpretation of the formal expression $Du/|Du|$. Other related contributions include Kawohl's study of a family of torsional creep problems \cite{Kawohl1990to} and Demengel's work on nonlinear partial differential equations involving the $1$-Laplacian and the critical Sobolev exponent \cite{Demengelf1999}. Degiovanni and Magrone \cite{Degio2009lin} studied a $1$-Laplacian version of the Brezis--Nirenberg problem by means of a nonsmooth linking argument. Molino Salas and Segura de Le\'on \cite{MolinoSal2018subc} used variational methods to construct nontrivial solutions of approximating $p$-Laplacian problems and then passed to the limit $p\downarrow1$. 

%{\color{red}Their model source is $|u|^{q-1}u$, with $0<q<1/(n-1)$; it is subcritical and satisfies $q>p-1$ for $p$ sufficiently close to $1$, which is the superlinearity condition relative to the $p$-Laplacian. Uniform $W_0^{1,1}(\Omega)$ bounds and a separate argument ensuring that the limits are nonzero yield the nontrivial solutions of their Theorem 1.1.}

Mercaldo, Segura de Le\'on, and Trombetti \cite{MercaldoSeguraTrombetti2008} studied the limit as $p \downarrow 1$ of weak solutions to Dirichlet $p$-Laplacian problems in bounded open domains $\Omega\subset\R^n$, $n\geq2$, with Lipschitz boundary and datum $f$. They considered separately the radial case $\Omega=B_R$, with nonnegative radially decreasing $f\in L^{n,\infty}(B_R)$, and the general case $f\in W^{-1,\infty}(\Omega)$. In the radial case, writing $f(x)=h(|x|)$, the solution has the representation
\begin{equation*}
v_p(r)=\int_r^R\left[s^{1-n}\int_0^s t^{n-1}h(t),\din t\right]^{1/(p-1)},\din s,\qquad 0<r<R.
\end{equation*}
The elementary limits $a^{1/(p-1)}\to0$ for $0\leq a<1$ and $a^{1/(p-1)}\to\infty$ for $a>1$, as $p  \downarrow 1$, explain the role of the value $1$ in the limiting behavior.

Molino and Segura de Le\'on \cite{MolinoSeguraLeon2022} studied the Gelfand-type problem
$$
\left\{\begin{aligned}
-\Delta_1 v&=\lambda f(v) && \text{in }\Omega,\\
v&=0 && \text{on }\partial\Omega,
\end{aligned}\right.
$$
where $\lambda\geq0$ and $f:[0,\infty)\to(0,\infty)$ is continuous, increasing, and unbounded, with $f(0)>0$. They proved the existence of a threshold $\lambda^*=h(\Omega)/f(0)$, where $h(\Omega)$ is the Cheeger constant of $\Omega$, such that there are no solutions when $\lambda>\lambda^*$, while the trivial function is always a solution when $\lambda\leq\lambda^*$.

De Cicco, Giachetti, and Segura de Le'on \cite{DeCiccoGiachettiSegura2019} obtained solutions of
$$
\left\{\begin{aligned}
-\Delta_1 u&=u^{-\gamma}f(x) && \text{in }\Omega,\\
u&=0 && \text{on }\partial\Omega,
\end{aligned}\right.
$$
as limits of solutions to the corresponding Dirichlet problems
$$
-\Delta_p u=u^{-\gamma}f(x)\qquad\text{in }\Omega.
$$
Here $\Omega\subset\R^n$ is a bounded open set with Lipschitz boundary, $0<\gamma\leq1$, and $0\leq f\in L^n(\Omega)$. For the more general source $h(u)f(x)$, see \cite{Marti2025opgl}, where $f\in L^m(\Omega)$ is nonnegative, with $m\geq1$, and $h:\mathbb{R}^+\to\mathbb{R}^+$ is continuous, possibly singular at the origin, and bounded at infinity.

Balducci \cite{Balducci2025} proved the existence of solutions to
$$
\left\{\begin{aligned}
-\Delta_1 u+g(u)|Du|&=h(u)f && \text{in }\Omega,\\
u&=0 && \text{on }\partial\Omega,
\end{aligned}\right.
$$
where $\Omega\subset\mathbb{R}^n$, $n\geq2$, is a bounded open set with Lipschitz boundary, $g$ is continuous and positive, possibly singular at the origin, and bounded at infinity, $h$ is continuous and nonnegative, possibly singular at the origin, and bounded at infinity, and $0\leq f\in L^n(\Omega)$. The existence theorem imposes additional compatibility conditions on the singular behavior of $g$ and $h$ near the origin. The case of continuous, bounded, and nonmonotone functions $g$ and $h$ is also covered in \cite{Balduc2024fin}.

\subsection{Main results}

Let $n\ge2$, {$0<\alpha\le1$}, and $1<p<\min\{2,\sqrt n\}$.
\begin{equation}\label{e6}
0\le\beta_p< p-1,\qquad
\ell^{-(n-p)/(p-1)}\le \xi  _1^2.
\end{equation}
The constants $\ell>0$, $\xi  _1>0$, and $\xi  _2>0$ are fixed and independent of \(p\).

The following assumptions are motivated by conditions (F1)--(F4) in \cite{MiaoYang2008}. Throughout, we assume that:
\begin{enumerate}[label=$(F_{\arabic*})$ ]
\item \label{1} $F_p\in C^{0,\alpha} _{\loc}([0,\infty)^3)$, $
\phi _p\in C^0([0,\infty)^2)$, and
\begin{equation}
0\le\phi _p(|x|,u) \le F_p(|x|,u,|Y|)
\label{e8}
\end{equation}
for all $x\in\mathbb R^n$,
$0<u\le\xi  _1$, and $
|Y|\le\xi  _2$.

\item \label{2}
For every \(r\ge0\), $u\mapsto\phi _p(r,u)$
 is nondecreasing on $[0,\xi  _1]$, and $(u,\rho)\mapsto F_p(r,u,\rho)$ is nondecreasing in each variable on  $[0,\xi  _1]\times[0,\xi  _2]$.

\item \label{3}
 
\begin{equation}
\xi  _1^{p-1-\beta_p}
\left(\frac{p-1}{n-p}\right)^{p-1}
\int_0^\ell t^{n-1}
{\phi _p\!\left(t,\xi  _1^{-1}\ftes _p    (t)\right)}\, \din t\ge1 ,
\label{e10}
\end{equation}
where $\ftes _p  (r):=\max\{\ell ,r\}^{-\frac{n-p}{p-1}}$.

\item \label{4}
\begin{equation}
\int_0^\infty
\left[
s^{1-n}\int_0^s t^{n-1}
F_p(t,\xi  _1,\xi  _2)\ftes _p    (t)^{-\beta_p}\, \din t
\right]^{1/(p-1)}ds
\le
\xi  _1^{\,1-\beta_p/(p-1)}.
\label{e11}
\end{equation}

\item \label{5}
For every \(s>0\),
\begin{equation}
s^{1-n}\int_0^s t^{n-1}
F_p(t,\xi  _1,\xi  _2)\ftes _p    (t)^{-\beta_p}\, \din t
\le
\xi  _2^{p-1}\xi  _1^{-\beta_p}.
\label{e12}
\end{equation}

\item \label{6} There is a constant
\(C>0\), independent of \(r\),
such that
\begin{equation}\label{eq:F6-bound}
F_p(r,\xi  _1,\xi  _2)\ftes _p    (r)^{-\beta_p}
<C , \qquad r\geq 0.
\end{equation}

\item \label{7} $ \int_0^\infty t^{n-1}
 F_p(t,\xi  _1,\xi  _2)
 \ftes _p    (t)^{-\beta_p}\, \din t<\infty$.
\end{enumerate}

\begin{remark}Conditions \ref{1} and \ref{7}, together with $\ell>0$ and $n>p^2$, imply, by Proposition \ref{10}, that
\begin{equation}\label{12}
\funct _{ \funct _A} (0) <\infty, 
\end{equation}
where
$$
\funct _A (r):=\int_r^{\infty}\left[s^{1-n} \int_0^s t^{n-1} A (t) \, \din t\right]^{\frac{1}{p-1}} \, \din s < \infty,
$$
and $A(r):= {F_p}(r , \xi _1 , \xi _2)(\xi _1 ^{-1} \ftes _p   (r)) ^{-\beta _p}$.

\end{remark} 

The first main result is the following.
\begin{theorem}\label{80}
Let $p_j\downarrow1$ and
\[
u_{p_j}(x)=v_{p_j}(|x|)
\]
be the positive radial solutions of \eqref{31} supplied by Proposition
\ref{25}. Assume \ref{1}--\ref{7} for $p=p_j$, with
$\ell,\xi_1,\xi_2$ independent of $j$, together with \eqref{32} and
\eqref{33}. Then, after extraction,
\begin{gather}
 u_{p_j}\to    u
 \quad\text{in }C_{\loc}(\mathbb R^n),
 \qquad
 \nabla u_{p_j}\stackrel{*}{\rightharpoonup}\nabla u
 \quad\text{in }L^\infty_{\loc}(\mathbb R^n;\mathbb R^n),
 \label{eq:80-compactness}\\
 u(x)=v(|x|),
 \qquad
 u\in W^{1,\infty}(\mathbb R^n),
 \qquad
 0\le u\le\xi_1,
 \qquad
 \|\nabla u\|_{L^\infty(\mathbb R^n)}\le\xi_2.
 \label{eq:80-u-bounds}
\end{gather}
Moreover, for some radial $w:\mathbb R^n\to(0,1]$,
\begin{gather}
 u_{p_j}^{p_j-1}\to    w
 \quad\text{a.e. and in }L^q_{\loc}(\mathbb R^n),
 \qquad 1\le q<\infty,
 \label{eq:80-w-convergence}\\
 \bigl(\max\{\ell,|x|\}\bigr)^{-(n-1)}
 \le w(x)\le1
 \quad\text{a.e.},
 \qquad
 w=1\quad\text{a.e. on }\{u>0\}.
 \label{eq:80-w-bounds}
\end{gather}
\begin{enumerate}[label=$(\roman*)$]
\item \label{111} There exist $\mathtt G\in W^{1,\infty}_{\loc}([0,\infty))$, $\mathtt K\in L^\infty_{\loc}([0,\infty))$, and ${\bf z}\in W^{1,\infty}_{\loc}(\mathbb R^n;\mathbb R^n)$ such that
\begin{gather}
 \mathtt G_{p_j}\to   \mathtt G
 \quad\text{in }C_{\loc}([0,\infty)),
 \qquad
 0\le\mathtt G\le1,
 \qquad
 \mathtt G(0)=0,
 \label{eq:80-G-limit}\\
 {\bf z}(x)=-\mathtt G(|x|)\frac{x}{|x|}
 \quad(x\ne0),
 \qquad
 |{\bf z}|\le1,
 \label{eq:80-z}\\
 -\operatorname{div}{\bf z}
 =\mathtt K(|x|),
 \qquad
 \mathtt K(r)=\mathtt G'(r)+\frac{n-1}{r}\mathtt G(r)
 \quad\text{for a.e. }r>0,
 \label{eq:80-divergence}\\
 ({\bf z},Du)=|Du|,
 \qquad
 {\bf z}\cdot\nabla u=|\nabla u|
 \quad\text{a.e. in }\mathbb R^n,
 \label{eq:80-calibration}\\
 (1-\mathtt G(r))|v'(r)|=0
 \quad\text{for a.e. }r>0.
 \label{eq:80-complementarity}
\end{gather}

\item \label{112} Set $\mathscr{C}_1:=\{r>0\mid\mathtt G(r)=1\}$ and $\mathscr{C}_<:=\{r>0\mid\mathtt G(r)<1\}$. Then,
\begin{equation}\label{95}
 \mathtt K(|x|)=
\left\{ \begin{aligned}
& w(x)^{-{\lbp}}F_1(|x|,u(x),0),
 & & {\text{for a.e. }x\in\mathbb R^n\text{ with }|x|\in\mathscr{C}_<,}\\
&  \dfrac{n-1}{|x|},
   & & {\text{for a.e. }x\in\mathbb R^n\text{ with }|x|\in\mathscr{C}_1,\ 
    |\mathscr{C}_1|>0}
 \end{aligned}\right. .
\end{equation}

Finally, there exist $r_0>0$ and  $ \delta\in(0,1)$ such that, for all sufficiently large $j$, 
\begin{equation}\label{eq:80-origin}
 \sup_{|x|\le r_0}|\nabla u_{p_j}(x)|
 \le \delta^{1/(p_j-1)},
\end{equation}
and
\begin{equation}\label{eq:80-origin-phase}
 \mathtt G(r)\le \delta<1
 \quad(0\le r\le r_0),
 \qquad
 \nabla u=0
 \quad\text{a.e. in }B_{r_0}.
\end{equation}
\end{enumerate}
\end{theorem}

For the second main result of this work, we additionally assume that
\begin{equation}\label{33}
F_p\to    F_1
\qquad\text{locally uniformly in }[0,\infty)^3
\quad\text{as }p\downarrow1,
\end{equation}
where $F_1\in C^{0,\alpha}_{\loc}([0,\infty)^3)$, and that
\begin{equation}\label{32}
\frac{\beta _p }{p-1}\to    {\lbp}\in[0,1].
\end{equation}

\begin{theorem}\label{81}
Assume the hypotheses and conclusions of Theorem \ref{80}. Then
\begin{equation}\label{eq:81-global-1laplace}
 -\Delta_1u=\mathtt K(|x|)
 \qquad\text{in }\mathbb R^n,
\end{equation}
in the sense of Definition \ref{113}.

\begin{enumerate}[label=$({\rm \roman*})$]
\item\label{thm81-affine}
 If, for every $(r,t)\in[0,\infty)\times[0,\xi_1]$, the map $\rho\mapsto   F_1(r,t,\rho)$ is affine on $ [0,\xi_2]$, then
\begin{equation*}\label{eq:81-affine-source}
 \mathtt K(|x|)
 =w(x)^{-{\lbp}}
 F_1\!\left(|x|,u(x),|\nabla u(x)|\right)
 \quad\text{for a.e. }x\in\mathbb R^n.
\end{equation*}

\item\label{thm81-critical}
One has $|v'(r)|=0$ for a.e.  $r\in\mathscr{C}_<$, and
\begin{equation*}\label{eq:81-outside-C-source}
 \mathtt K(|x|)
 =w(x)^{-{\lbp}}
 F_1\!\left(|x|,u(x),|\nabla u(x)|\right)
 \quad\text{for a.e. }x\in\mathbb R^n
 \text{ with }|x|\in\mathscr{C}_<.
\end{equation*}

Suppose that ${E\subset\mathscr C_1}$ is measurable with $|E|>0$, and assume that
\begin{equation*}
 a(r):=
 \lim_{j\to\infty}
 \frac{\mathtt G_{p_j}(r)-1}{p_j-1}
 \in\mathbb R
 \qquad\text{for a.e. }r\in E.
\end{equation*}

Then
\begin{equation*}\label{eq:81-gradient-E}
 |v'(r)|=e^{a(r)}
 \qquad\text{for a.e. }r\in E,
\end{equation*}
and
\begin{equation*}\label{eq:81-source-E}
 \mathtt K(|x|)
 =w(x)^{-{\lbp}}
 F_1\!\left(|x|,u(x),|\nabla u(x)|\right)
 =\frac{n-1}{|x|}
 \quad\text{for a.e. }x\in\mathbb R^n
 \text{ with }|x|\in E.
\end{equation*}
Hence
\begin{equation*}\label{eq:81-identification-union}
 \mathtt K(|x|)
 =w(x)^{-{\lbp}}
 F_1\!\left(|x|,u(x),|\nabla u(x)|\right)
 \quad\text{for a.e. }x\in\mathbb R^n
 \text{ with }
 |x|\in\mathscr{C}_<\cup E.
\end{equation*}

\end{enumerate}
\end{theorem}

%\fragmento{frag1}

 %%
 %%
 %%
 %%
 %%
 %%
%%
 %%
 %%
 %%
 %

 \section{Preliminaries}

In this article, we use the following notation:
\begin{enumerate}[label=$(\roman*)$]
\item For every measurable set \(E \subset \mathbb{R}^m\), we denote its Lebesgue measure by $\mathcal{L}^m(E)=|E|$.

\item $\qpu:=\frac{n-p}{p-1}$.

\item We write
$$
\lbp:=\lim_{p\to 1}\frac{\beta_p}{p-1}\in[0,1]
$$

whenever the limit exists.

\item For $h\in L^0(\R)$ with $g\geq 0$, we define
$$
\funct_h(r):=\int_r^{\infty}\left[s^{1-n}\int_0^s t^{n-1}h(t)\,\din t\right]^{\frac{1}{p-1}}\,\din s.
$$

\item For $\ell,t\geq 0$, we define
$\ftes(t):=\max{\ell,t}$ and
$\ftes_p(t):=\max{\ell,t}^{-\frac{n-p}{p-1}}$.

\item For radial solutions $u_p(x)=v_p(|x|)$ of \eqref{31}, we use the notation (see page \pageref{115})
$$
\begin{gathered}
\mathtt{K}_p(r):=F_p\left(r,v_p(r),\left|v_p^{\prime}(r)\right|\right)v_p(r)^{-\beta_p},\\
\mathtt{G}_p(r):=r^{1-n}\int_0^r s^{n-1}\mathtt{K}_p(s)\,\din s
\quad (r>0), \quad \mathtt{G}_p(0):=0,\\
{\bf z}_p:=\left|\nabla u_p\right|^{p-2}\nabla u_p.
\end{gathered}
$$

\item For radial solutions $u_p(x)=y_p(|x|)$ of \eqref{117}, in the case $n=1$, we use the notation (see page \pageref{116})
\begin{gather*}
k_p(t):=\fnu_p\left(t,y_p(t),y_p^{\prime}(t)\right)y_p(t)^{-\beta_p},\\
G_p(t):=\int_0^t k_p(s),\din s=y_p^{\prime}(t)^{p-1},\\
z_p:=|u_p^\prime|^{p-2}u_p^\prime.
\end{gather*}

\item Let $\Omega\subset\mathbb R^n$ be open and let $f_j,f\in L^\infty_{\mathrm{loc}}(\Omega;\mathbb R^m)$. We write $f_j\stackrel{*}{\rightharpoonup}f$ in $L^\infty_{\mathrm{loc}}(\Omega;\mathbb R^m)$ if
\[
\int_U f_j(x)\cdot\varphi(x)\,\din x
\longrightarrow
\int_U f(x)\cdot\varphi(x)\,\din x
\]
for every open set $U\Subset\Omega$ and every $\varphi\in L^1(U;\mathbb R^m)$. 

\item For $[a_{ij}],[b_{ij}]\in\mathbb{R}^{n\times n}$, we write $[a_{ij}]:[b_{ij}]:=\sum_{i=1}^n\sum_{j=1}^n a_{ij}b_{ij}$.
\end{enumerate}

%%%%
%%
%%
%%
%%%
%%

 \subsection{Basic properties of functions of bounded variation}
\label{120}

We use the following measure-theoretic notation for the limit $p\downarrow1$. The basic references are
\cite{AmbrosioFuscoPallara2000, EvansGariepy2015}. We use the pairing between bounded
fields and $BV$ derivatives defined in
\cite{Anzellotti1983,ChenFrid1999,CrastaDeCicco2019}.

Throughout this subsection, $U\subset\mathbb R^n$ is open.
\begin{definition} Let $u \in L^1(U)$. We say that $u$ is a function of bounded variation in $U$ if its distributional derivative is represented by a finite Radon measure in $U$; that is,
\begin{equation*}
\int_{U} u  \phi_{ x_i} \, \din x=-\int_{U} \phi \, \din D_i u \quad \forall \phi \in C_c^{\infty}(U) . \quad i=1, \ldots, n 
\end{equation*}
for some $\R^n$-valued measure $D u=\left(D_1 u \ldots D_n u\right)$ in $U$. The vector space of all functions of bounded variation in $U$ is denoted by $B V(U)$.
\end{definition}

We denote
\begin{equation*}\label{eq:BVloc-definition-prelim}
BV_{\mathrm{loc}}(U)
:=
\left\{
 u\in L^1_{\mathrm{loc}}(U) \mid
 u\in BV(V)\ \text{for every }V\Subset U
\right\}.
\end{equation*}

For $u\in BV(U)$, define
\[
\|u\|_{BV(U)}:=\|u\|_{L^1(U)}+|Du|(U).
\]

We denote by $\mathcal{M}_{\loc}\left(U ; \mathbb{R}^m\right)$ the space of  $\mathbb{R}^m$-valued Radon measures on $U$. For $\mu \in \mathcal{M}_{\loc}\left(U ; \mathbb{R}^m\right)$, its total variation $|\mu| \in \mathcal{M}_{\loc}(U)$ is defined on every  {relatively compact} Borel set $B\subset U$ by
$$
|\mu|(B):=\sup \left\{\sum_{j=1}^{\infty}\left|\mu\left(B_j\right)\right| \mid B_j \subset U \text { Borel }, \quad B=\bigcup_{j=1}^{\infty} B_j, \quad B_i \cap B_j=\varnothing \text { for } i \neq j\right\} .
$$
For an arbitrary Borel set $B\subset U$, set $|\mu|(B):=\sup_{V\Subset U}|\mu|(B\cap V)$.

The polar decomposition \cite[Corollary~1.29]{AmbrosioFuscoPallara2000} gives {$\sigma\in[L^\infty(U,|\mu|)]^m$} such that 
\begin{equation}\label{121}
\left|\sigma\right|=1 \quad|\mu| \text {-a.e., } \quad
\mu=|\mu| \sigma.
\end{equation}
If $|\mu|(U)<\infty$, then $\sigma\in[L^1(U,|\mu|)]^m$.

By \cite[Section~5.1]{EvansGariepy2015}, if $u\in BV_{\loc}(U)$, then 
\begin{equation*}\label{eq:BV-variation-dual-prelim}
|Du|(V)
=
\sup\left\{
\int_V u\,\operatorname{div}\varphi\,\din x \mid 
\varphi\in C_c^1(V;\mathbb R^n),
\quad
|\varphi(x)|\le1
\right\}.
\end{equation*}
Here $V\Subset U$ is arbitrary.

For $u\in BV_{\loc}(U)$, Lebesgue's decomposition theorem gives
$$
D_i u = ( D_i u)_{\mathrm{ac}}+ ( D_i u)_{\mathrm{s}}
$$
where
$$
( D_iu)_{\mathrm{ac}} \ll \mathcal{L}^n, \quad ( D_i u)_{\mathrm{s}} \perp \mathcal{L}^n .
$$
Hence
$$
|D u| = |D u|_{\mathrm{ac}}+ | D u|_{\mathrm{s}}
$$
where $|D u|_{\mathrm{ac}} = |(( D_1 u)_{\mathrm{ac}} , \ldots , ( D_n u)_{\mathrm{ac}})|$ and $|D u|_{s} = |(( D_1 u)_{\mathrm{s}} , \ldots , ( D_n u)_{\mathrm{s}})|$, see \cite[Lebesgue decomposition theorem]{zbMATH03576139}.

For each $i=1,\ldots,n$, the Radon--Nikod\'ym theorem gives $f_i\in L^1_{\loc}(U)$ such that
$$
(D_i u)_{\mathrm{ac}} = \mathcal{L}^n \llcorner f_i
$$
that is,
$$
 (D_i u)_{\mathrm{ac}} (K)= \int _K  f_i \, \din x
$$
for every compact set $K\subset U$.

Let $\sigma$ be given by \eqref{121} with $\mu=Du$. Since we will consider functions $u\in W^{1,\infty}(\R^n)$, we recall that $W^{1,1}_{\loc}(U)\subset BV_{\loc}(U)$. If $u\in W^{1,1}_{\loc}(U)$, then  
$$
f_i = u_{x_i} \text { a.e } \qquad |Du|_s =0 .
$$
Furthermore, \cite[Section~5.1]{EvansGariepy2015} gives
\begin{equation}\label{122}
|Du|=\mathcal{L}^n\llcorner|\nabla u| , 
\end{equation}
and, we have $\mathcal{L}^n$-a.e.
$$
\sigma= \left\{ \begin{aligned}&\frac{\nabla u }{|\nabla u |} & & \text { if } \nabla u  \neq 0 \\ 
&0 &  &\text { if } \nabla u =0 .\end{aligned} \right.
$$

\medskip

\noindent {\bf Divergence-measure fields and the Anzellotti pairing.} We use the class
\begin{equation*}\label{eq:DMinfty-prelim}
\mathscr{X}(U)
:=
\left\{
{\bf z}\in L^\infty_{\mathrm{loc}}(U;\mathbb R^n) \mid
\operatorname{div}{\bf z}\in L_{\loc}^{\infty}(U)
\right\},
\end{equation*}
where $\operatorname{div}{\bf z}$ is the distribution defined by
$\left\langle \operatorname{div}{\bf z}, \zeta \right\rangle = - \int _U {\bf z} \cdot \nabla \zeta \, \din x $ for $\zeta\in C_c^\infty(U)$.

Let ${\bf z}\in \mathscr{X}(U)$ and $u\in BV_{\mathrm{loc}}(U)\cap L^\infty_{\mathrm{loc}}(U)$. The Anzellotti pairing is
defined distributionally by
\begin{equation*}\label{eq:Anzellotti-pairing-prelim}
\langle({\bf z},Du),\zeta \rangle
=
-\int_U u\zeta \operatorname{div}{\bf z} \,\din x
-\int_Uu\,{\bf z}\cdot\nabla\zeta \,\din x,
\qquad
\zeta \in C_c^\infty(U).
\end{equation*}

 The pairing is a Radon measure and satisfies, for every $V\Subset U$ and every Borel set $B\subset V$,
\begin{equation*}\label{eq:Anzellotti-estimate-prelim}
|({\bf z},Du)|(B)
\le
\|{\bf z}\|_{L^\infty(V;\mathbb R^n)}|Du|(B);
\end{equation*}
see the proof of \cite[Theorem~3.3]{CrastaDeCicco2019}.

\begin{definition}\label{113}
For $u\in BV_{\mathrm{loc}}(\mathbb R^n)\cap L^\infty_{\mathrm{loc}}(\mathbb R^n)$ and $f\in{L_{\loc}^\infty}(\mathbb R^n)$, we write
\[
-\Delta_1u=f
\]
if there exists ${\bf z}\in \mathscr{X}(\mathbb R^n)$ such that
\begin{equation*}\label{eq:1laplace-definition-43}
-\operatorname{div}{\bf z}=f ,
\qquad |{\bf z}|\leq1\quad\mathcal L^n\text{-a.e.},
\qquad ({\bf z},Du)=|Du|.
\end{equation*}
{The first equality holds in the sense of distributions, and the last equality holds in the sense of Radon measures.}
\end{definition}

 \subsection{An iterated radial integral estimate}

For a measurable $h:[0,\infty)\to[0,\infty)$ and $r\ge0$, define
\begin{equation}
\funct_{h}(r)
:=
\int_r^{\infty}
\left[
 s^{1-n}\int_0^s t^{n-1}h(t)\, \din t
\right]^{{\qpu}} \, \din s, \qquad {\qpu}:=\frac{1}{p-1}.
\label{eq:Phi}
\end{equation}

\begin{proposition}\label{10}
Let $1\le m<\infty$, $\alpha\in\mathbb R$, and $\theta>0$.
There exists $C>0$, independent of $h$, such that
\begin{equation}
\funct_{\funct_{h}}(0)
\le
C\left(
\int_0^{\infty} t^{\alpha}h(t)^m\,\din t
\right)^{\theta}
\label{eq:main-general}
\end{equation}
for every measurable $h:[0,\infty)\to[0,\infty)$ with finite right-hand side if and only if
\begin{equation}
 n>p^2,\qquad
 1\le m\le {\qpu} ^2,\qquad
 \alpha=mp^2-1,\qquad
 \theta=\frac{{\qpu} ^2}{m}.
\label{eq:conditions}
\end{equation}

Under \eqref{eq:conditions}, define $\rho\in[1,{\qpu} ^2]$ by
\begin{equation}
\frac{1}{\rho}
=
1+\frac{1}{{\qpu} ^2}-\frac{1}{m}.
\label{eq:rho}
\end{equation}
Then one may take
\begin{equation}
C=
\bigl[p(n-p)\bigr]^{- {\qpu}}
\bigl[\rho(n-p^2)\bigr]^{-  {\qpu}^2/\rho}.
\label{eq:constant}
\end{equation}

\end{proposition}

\begin{proof}
We first prove sufficiency. Assume
\[
n>p^2,
\qquad
1\le m\le {\qpu} ^2.
\]

Let
\[
k_a(x):=e^{-ax}\mathbf 1_{[0,\infty)}(x), 
\qquad
j_p(x):=e^{px}\mathbf 1_{(-\infty,0]}(x),
\]
where $a:=n-p^2>0$.

Define
\[
f(x):=e^{p^2x}h(e^x),
\qquad x\in\mathbb R.
\]

A direct calculation gives
\begin{align}
&e^{(1-n)x}\int_0^{e^x}t^{n-1}h(t)\,\din t
\notag\\
&\qquad
=e^{(1-p^2)x}
\int_{-\infty}^{x}e^{-a(x-y)}f(y)\, \din y
=e^{(1-p^2)x}(k_a*f)(x).
\label{eq:first-average}
\end{align}

Since ${\qpu} (1-p^2)=-(p+1)$,
\begin{align}
\funct_{h}(e^x)
&=
\int_x^{\infty}
 e^{-pz}(k_a*f)(z)^{{\qpu}}\,\din z
\notag\\
&=
e^{-px}G(x),
\label{eq:first-iterate-log}
\end{align}
where
\begin{equation*}
G(x)
:=
\int_0^{\infty}e^{-p\sigma}(k_a*f)(x+\sigma)^{{\qpu}}\, \din \sigma
=
j_p*\bigl[(k_a*f)^{{\qpu}}\bigr](x).
\label{eq:G}
\end{equation*}

Using
\eqref{eq:first-iterate-log},
\begin{align}
&e^{(1-n)x}\int_0^{e^x}t^{n-1}\funct_{h}(t)\, \din t
\notag\\
&\qquad
=e^{-(p-1)x}
\int_{-\infty}^{x}e^{-b(x-y)}G(y)\, \din y
=e^{-(p-1)x}(k_b*G)(x),
\label{eq:second-average}
\end{align}
where
$$
k_b(x):=e^{-bx}\mathbf 1_{[0,\infty)}(x), \qquad b:=n-p>0.
$$

With the change of variables $s=e^x$, \eqref{eq:second-average} gives
\begin{equation} \label{eq:exact-log-representation}
\begin{aligned}
\funct_{\funct_{h}}(0)
= &\int _0 ^\infty \left[ s^{1-n} \int ^s_0 t^{n-1}\funct_h{(t)}\, \din t \right]^{{\qpu}} \, \din s\\
= & \int_{-\infty}^{\infty}(k_b*G)(x)^{{\qpu}}\,\din x\\
= & \|k_b*G\|_{L^{{\qpu}}(\mathbb R)}^{{\qpu}}.
\end{aligned}
\end{equation}

By \eqref{eq:rho} and $1\le m\le   {\qpu}^2$, one has $1\le\rho\le {\qpu} ^2$ and
\[
1+\frac1{{\qpu}  ^2}=\frac1m+\frac1\rho.
\]

Young's convolution inequality $\|f_1*f_2\|_{L^{p_1}}\leq\|f_1\|_{L^{p_2}}\|f_2\|_{L^{p_3}}$, where $1+1/p_1=1/p_2+1/p_3$, gives
\begin{equation}
\|k_a*f\|_{L^{{\qpu} ^2} (\R)}
\le
\|k_a\|_{L^{\rho} (\R) }\|f\|_{L^m (\R) }.
\label{eq:young1}
\end{equation}
Moreover,
\[
\|j_p\|_{L^1 (\mathbb{R})}=\frac1p,
\qquad
\|k_b\|_{L^1 (\mathbb{R}) }=\frac1{n-p},
\qquad
\|k_a\|_{L^{\rho} (\mathbb{R}) }
=\bigl[\rho(n-p^2)\bigr]^{-1/\rho}.
\]

Hence
\begin{align}
\|G\|_{L^{{\qpu}} (\mathbb{R})}
&\le
\|j_p\|_{L^1 (\mathbb{R})}
\|(k_a*f)^{{\qpu}}\|_{L^{{\qpu}} (\mathbb{R})}
\notag\\
&=
\frac1p\|k_a*f\|_{L^{{\qpu} ^2} (\mathbb{R})}^{{\qpu}}  ,
\label{eq:young2}
\end{align}
and then
\begin{align}
\funct_{\funct_{h}}(0)
&=
\|k_b*G\|_{L^{{\qpu}} (\mathbb{R})}^{{\qpu}}
\notag\\
&\le
\|k_b\|_{L^1 (\mathbb{R}) }^{{\qpu}}\|G\|_{L^{{\qpu}} (\mathbb{R})}^{{\qpu}}
\notag\\
&\le
\bigl[p(n-p)\bigr]^{- {{\qpu}}}
\|k_a*f\|_{L^{{{\qpu}}^2} (\mathbb{R})}^{{{\qpu}} ^2}
\notag\\
&\le
\bigl[p(n-p)\bigr]^{-  {{\qpu}}}
\bigl[\rho(n-p^2)\bigr]^{-  {{\qpu}}^2/\rho}
\|f\|_{L^m (\mathbb{R})}^{  {{\qpu}}^2}.
\label{eq:final-young}
\end{align}

With the change of variables $r=e^x$,
\begin{equation*}
\|f\|_{L^m(\mathbb R)}^m
=
\int_{-\infty}^{\infty}e^{mp^2x}h(e^x)^m\,\din x
=
\int_0^{\infty}r^{mp^2-1}h(r)^m\,  \din r.
\label{eq:f-norm}
\end{equation*}
This proves \eqref{eq:main-general} under \eqref{eq:conditions}.

\medskip

We now prove necessity.

\medskip

\noindent {\bf Step 1.} We prove that $n>p^2$ is necessary. Set
$$
h:=\mathbf{1}_{(1,2)}.
$$
If $n\le p$, then $\funct_{h}(r)=\infty$ for $r\geq 3$, so
\eqref{eq:main-general} is impossible.

Assume $n>p$. Let
\[
c_1:=\int_0^{\infty}t^{n-1}h(t)\, \din t>0.
\]
For $r\ge2$,
\begin{equation}
\funct_{h}(r)
=
c_1^{{\qpu}}\int_r^{\infty}s^{-  {{\qpu}}(n-1)}\, \din s
={\frac{p-1}{n-p}}c_1^{{\qpu}} r^{-\frac{n-p}{p-1}}.
\label{eq:tail-first}
\end{equation}

Put
\[
c_2:=\frac{n-p}{p-1}.
\]

If $n\le p^2$, then $c_2\le p<n$. Consequently, for $s\geq {2^{1+1/(n-c_2)}}$,
\[
\int_0^s t^{n-1}\funct_{h}(t)\, \din t
\geq {\frac{p-1}{n-p}}c_1^ {{\qpu}}  \int_2^s t^{n-1-c_2}\, \din t
\geq {\frac{(p-1)c_1^ {{\qpu}}}{2(n-p)(n-c_2)}}s^{n-c_2},
\]
and therefore
\[
\funct_{\funct_{h}}(0)
\ge
c_3\int^{\infty}_{{2^{1+1/(n-c_2)}}}s^{-  {{\qpu}}(c_2-1)}\, \din s.
\]
Here \({c_3:=\left[\frac{(p-1)c_1^ {{\qpu}}}{2(n-p)(n-c_2)}\right]^ {{\qpu}}}\).
Since $c_2\le p=1+1/{{\qpu}}$, one has ${{\qpu}}(c_2-1)\le1$, and the last integral diverges.
Thus $n>p^2$ is necessary.

\medskip

\noindent {\bf Step 2.} Assume that \eqref{eq:main-general} holds. We prove that
\begin{equation}
\theta=\frac{{{\qpu}}^2}{m},
\qquad
\alpha=mp^2-1.
\label{91}
\end{equation}

For \({h=\mathbf 1_{(1,2)}}\),
$0<\funct_{\funct_{h}}(0)<\infty$.
For $\sigma >0$,
\[
\funct_{\funct_{\sigma h}}(0)=\sigma ^{{{\qpu}}^2}\funct_{\funct_{h}}(0).
\]
Applying \eqref{eq:main-general} to $\sigma h$ and varying $\sigma$ gives
\begin{equation}
m\theta={{\qpu}}^2.
\label{eq:amplitude-scaling}
\end{equation}

Next, for $\lambda>0$, define $h_{\lambda}(t):=h(\lambda t)$. A
change of variables gives
\begin{equation}
\funct_{h_{\lambda}}(r)
=\lambda^{-pq}\funct_{h}(\lambda r),
\qquad
\funct_{\funct_{h_{\lambda}}}(0)
=
\lambda^{-p^2q^2}\funct_{\funct_{h}}(0).
\label{eq:dilation-left}
\end{equation}

On the other hand,
\[
\left(
\int_0^{\infty}t^{\alpha}h_{\lambda}(t)^m\, \din t
\right)^{\theta}
=
\lambda^{-(\alpha+1)\theta}
\left(
\int_0^{\infty}t^{\alpha}h(t)^m\, \din t
\right)^{\theta}.
\]
Varying $\lambda$ in \eqref{eq:main-general} and using
\eqref{eq:amplitude-scaling} and  \eqref{eq:dilation-left} yields
\[
(\alpha+1)\theta=p^2q^2,
\]
and therefore \eqref{91} holds.

\medskip

\noindent {\bf Step 3.} We prove that $m\le {{\qpu}}^2$ is necessary. Assume that \eqref{91} holds and
$n>p^2$. For $\sigma >4$, define
\begin{equation*}
h_\sigma (r):=
\sigma^{-1/m}r^{-p^2}\mathbf 1_{[1,e^\sigma]}(r).
\label{eq:test-hN}
\end{equation*}

Then
\begin{equation}
\int_0^{\infty}r^{mp^2-1}h_\sigma (r)^m\,dr=1.
\label{eq:test-norm}
\end{equation}

Let $a=n-p^2>0$ and $b=n-p>0$. For $1\le x\le \sigma-1$,
\[
(k_a*f_\sigma)(x)
\ge
\sigma ^{-1/m}\int_{x-1}^{x}e^{-a(x-y)}\, \din y
=c_4\sigma ^{-1/m},
\]
where $c_4= (1-e^{-a})/a$, and
\[
k_a(x)=e^{-a x} \mathbf{1}_{[0, \infty)}(x), \qquad f_\sigma (x)= e^{p^2 x} h_\sigma (e^x)=\sigma ^{-1/m}\mathbf 1_{[0,\sigma]}(x).
\]

Hence, for $1\le x\le \sigma-2$,
\begin{align*}
G_\sigma(x):= & \int_0^{\infty} e^{-p t}\left(k_a * f_\sigma\right)(x+t)^{{\qpu}} \, \din t\\
\ge & 
\int_0^1e^{-pt}(k_a*f_\sigma)(x+t)^  {{\qpu}} \, \din t \\
\ge & c_4^ {{\qpu}} e^{-p}\sigma ^{-   {{\qpu}}/m}.
\end{align*}

For $2\le x\le \sigma -2$,
\begin{align*}
(k_b*G_\sigma )(x)
\ge &
\int_{x-1}^{x}e^{-b(x-y)}G_\sigma (y)\, \din y\\
\ge & c_5 \sigma ^{-  {{\qpu}}/m},
\end{align*}
where $c_5 = c_4^{{\qpu}} e^{-p} (1-e^{-b})b^{-1}$.

By \eqref{eq:exact-log-representation},
\begin{equation*}
\funct_{\funct_{h_\sigma}}(0)
\ge
{c_5^  {{\qpu}}}(\sigma-4)\sigma^{-   {{\qpu}}^2/m}.
\label{eq:test-growth}
\end{equation*}
If $m> {{\qpu}}^2$, the right-hand side tends to
$\infty$, whereas \eqref{eq:test-norm} is identically equal to $1$.
Thus no uniform constant in \eqref{eq:main-general} can exist when
$m>  {{\qpu}}^2$.

This proves necessity and completes the proof.
\end{proof}

\begin{remark}
The proof of Proposition~\ref{10} uses the exponential substitution $s=e^x$ to obtain convolution estimates on $\mathbb R$. For properties of the isometry $(W_p f)(t)=e^{-t/p}f(e^{-t})$, $t\in\mathbb R$, from $L^p((0,\infty))$ onto $L^p(\mathbb R)$, and its role in reducing Hardy inequalities to convolution inequalities, see  \cite[Section~1.3.1, p.~15]{zbMATH06482033}  and \cite[Section~5.1, p.~52]{zbMATH01614151}.
\end{remark}

%{\color{orange}The continuity estimate for $T$ uses the following bound for the power map $s\mapsto s^{1/(p-1)}$, $1<p\le2$.}

%%
%%
%%
%%
%%
%%
%%
%%

 \section{Existence of radial solutions}\label{110}

For fixed $p$, we express the radial form of equation \eqref{31} as an integral fixed-point problem. Positivity, decay, and bounds on the derivative follow from explicit barriers, while compactness in $C_{\mathrm{loc}}^1([0,\infty))$ yields a solution. These estimates provide the family used in the limit analysis as $p\downarrow1$.

% ============================================================
% Photo 1
% ============================================================

We define
\[
\Lambda
:=
\left\{
h\in L^0 ([0,\infty))
\;\middle|\;
\int_{0}^{\infty}
\left[
s^{1-n}
\int_{0}^{s} t^{n-1}h(t)\, \din t
\right]^{\frac{1}{p-1}}
\, \din s
<\infty,
\quad
h\geq 0
\right\}.
\]

We also define
\[
\funct :\Lambda\to    C([0,\infty))
\]
by
\[
\funct _h(r)
:=
\int_{r}^{\infty}
\left[
s^{1-n}
\int_{0}^{s} t^{n-1}h(t)\, \din t
\right]^{\frac{1}{p-1}}
\, \din s.
\]

We begin with the following two technical results.
\begin{lemma} \label{16} 
$ $
\begin{enumerate}[label=$(\roman*)$]
    \item  \label{13} Suppose \(h_{1},h_{2}\in\Lambda\). If \(h_{1}\leq h_{2}\) a.e., then
    \[
    \funct _{h_{1}}\leq \funct _{h_{2}}
    .
    \]

    \item \label{14} If \(h\in\Lambda\), then, {for every \(r\geq0\)},
    \[
    \frac{p-1}{n-p}
    \left(
        \int_{0}^{{\ell}} t^{n-1}h(t)\, \din t
    \right)^{\frac{1}{p-1}}
    \varphi_{p}(r)
    \leq
    \funct _h(r)
    \]
    \[
    \leq
    \int_{0}^{\infty}
    \left(
        s^{1-n}
        \int_{0}^{s} t^{n-1}h(t)\, \din t
    \right)^{\frac{1}{p-1}}
    \, \din s.
    \]
\end{enumerate}

\end{lemma}

\begin{proof}

\noindent \ref{13} This follows directly from the definition of \(\funct \).

\medskip

\noindent \ref{14} For \(r\geq {\ell}\),
    \begin{align*}
    \funct _h(r)
    &=
    \int_{r}^{\infty}
    \left[
        s^{1-n}
        \int_{0}^{s} t^{n-1}h(t)\, \din t
    \right]^{\frac{1}{p-1}}
    \, \din s
    \\
    &\geq
    \int_{r}^{\infty}
    \left[
        s^{1-n}
        \int_{0}^{{\ell}} t^{n-1}h(t)\, \din t
    \right]^{\frac{1}{p-1}}
    \, \din s
    \\
    &\geq
    \left(
        \int_{0}^{{\ell}} t^{n-1}h(t)\, \din t
    \right)^{\frac{1}{p-1}}
    \int_{r}^{\infty}
    s^{\frac{1-n}{p-1}}
    \, \din s
    \\
    &=
    \frac{p-1}{n-p}
    \left(
        \int_{0}^{{\ell}} t^{n-1}h(t)\, \din t
    \right)^{\frac{1}{p-1}}
    r^{-\frac{n-p}{p-1}}.
    \end{align*}

    For \(r<{\ell}\),
     \begin{align*}
    \funct _h(r)
    \geq &
    \int_{{\ell}}^{\infty}
    \left[
        s^{1-n}
        \int_{0}^{s} t^{n-1}h(t)\, \din t
    \right]^{\frac{1}{p-1}}
    \, \din s\\
    \geq &
    \left(
        \int_{{\ell}}^{\infty}
        s^{\frac{1-n}{p-1}}
        \, \din s
    \right)
    \left(
        \int_{0}^{{\ell}}
        t^{n-1}h(t)\, \din t
    \right)^{\frac{1}{p-1}}
    \\
    = &
    \frac{p-1}{n-p}
    {\ell}^{-\frac{n-p}{p-1}}
    \left(
        \int_{0}^{{\ell}}
        t^{n-1}h(t)\, \din t
    \right)^{\frac{1}{p-1}}.
    \end{align*}
\end{proof}

\begin{lemma} \label{15}
Let $a,b\geq0$ and $1<p\leq2$. Then
\[
\left|
a^{\frac{1}{p-1}}
-
b^{\frac{1}{p-1}}
\right|
\leq
\frac{1}{p-1}
|a-b|
\left(
a^{\frac{2-p}{p-1}}
+
b^{\frac{2-p}{p-1}}
\right).
\]
\end{lemma}

\begin{proof}
Assume that  $0\leq a<b$ and $t>1$. By the mean value theorem, there exists $c\in(a,b)$ such that
\[
b^ {{t }} -a^{{t }}
=
{{t }} c^{{{t }}  -1}(b-a)
\leq
{{t }} b^{  {{t }} -1}(b-a).
\]

Hence,
\[
|a^{{t }}  -b^{{t }}|
\leq
{{t }}\left(a^{{{t }} -1}+b^{ {{t }} -1}\right)|a-b|,
\]
for $a,b\geq0$.

This proves the result.\end{proof}

We define
\[
Q
:=
\left\{
g\in\Lambda\cap C^{1}([0,\infty))
\;\middle|\;
\xi_{1}^{-1}\varphi_{p}(r)
\leq g(r)\leq \xi_{1},
\quad
|g'(r)|\leq \xi_{2},
\quad
r\geq 0
\right\}.
\]

The set \(Q\) is nonempty and convex. Indeed,
\[
\xi_1^{-1}\ell^{-c_0}\left[1+\left(\frac{\max\{r-\ell,0\}}{c_2}\right)^2\right]^{-c_1/2} \in Q
\]
where  $c_0 = (n-p)/(p-1)$, $ c_1\in\left(p,\min\left\{c_0,n\right\}\right)$, and $
 c_2 \geq\max\left\{\ell,\xi_1^{-1} \xi _2^{-1}\ell^{-c_0} c_1 \right\}$.

% ============================================================
% Photo 4
% ============================================================

%%
%%
%%
%%

We define $\mathcal{K}:Q\to   \Lambda$ by
\[
\mathcal{K}g(r)
:=
F_{p}\bigl(r,g(r),|g'(r)|\bigr)
g(r)  ^{-\beta_{p}} .
\]
This map is well-defined by \ref{2} and \ref{4}.

Now we define $T:Q\to    C^{2}([0,\infty)) $ by
\[
Tg:=\funct _{\mathcal{K}g}.
\]

On \(C^{1}([0,\infty))\), we consider the topology generated by
the seminorms
$$
\mathbf{s}_m(u):=\sup _{x \in[0, m]}\left(|u(x)|+\left|u^{\prime}(x)\right|\right), \quad m=1,2, \ldots
$$

\begin{proposition}\label{25}
 Assume that $F_p \in C^{0, \alpha} _{\loc} \left([0, \infty)^3\right)$ and $\phi_p \in C^{0}\left([0, \infty)^2\right)$ satisfy \ref{1}--\ref{7}. Then there exists $g_0\in Q$ such that
$$
Tg_0=g_0.
$$

Equivalently,
$$
g_0(r)=\int_r^{\infty}\left[s^{1-n} \int_0^s t^{n-1} {F_p}\left(t, g_0(t),\left|g_0^{\prime}(t)\right|\right) g_0(t)^{-\beta _p} \, \din t\right]^{{\qpu}} \, \din s
.$$
\end{proposition}
\begin{proof}

Let \(g\in Q\). From \ref{2},
\begin{equation} \label{17}
\begin{aligned}
\mathcal{K}g (r)
&\leq
F_{p}(r,\xi_{1},\xi_{2})
\left(
    \xi_{1}^{-1}\varphi_{p}(r)
\right)^{-\beta_{p}}
\\
&=:A(r), \qquad r\geq 0.
\end{aligned}
\end{equation}

\noindent {\bf Part 1.} We show that \(T\) is continuous.

{Let \(g_0\in Q\), let \(\{g_i\}_{i=1}^{\infty}\subset Q\), and suppose that \(g_i\to g_0\) in \(C^1_{\loc}([0,\infty))\).} Let $R>0$ and \(\epsilon\in(0,1)\). We have
\begin{equation}\label{22}
\begin{aligned}
|Tg_{i}(r)-Tg_{0}(r)|
\leq
\Bigg|
&
\int_{r}^{\infty}
\left[
    \int_{0}^{s}
    \left(\frac{t}{s}\right)^{n-1}
    \mathcal{K}g_{i}(t)\, \din t
\right]^{\frac{1}{p-1}}
\, \din s
 \\
&-
\int_{r}^{\infty}
\left[
    \int_{0}^{s}
    \left(\frac{t}{s}\right)^{n-1}
    \mathcal{K}g_{0}(t)\, \din t
\right]^{\frac{1}{p-1}}
\, \din s
\Bigg|.
\end{aligned}
\end{equation}

% ============================================================
% Photo 7
% ============================================================

By \eqref{17} and \ref{4}, there exists \(m_{0}>0\) such that
\begin{equation}\label{23}
\int_{r}^{\infty}
\left[
    \int_{0}^{s}
    \left(\frac{t}{s}\right)^{n-1}
    {\mathcal{K}g_{i}(t)}\, \din t
\right]^{\frac{1}{p-1}}
\, \din s
<
\epsilon,
\end{equation}
for $r\geq m_{0}$ and {\(i=0,1,2,\ldots\)}.

Since \(F_p\) is H\"older continuous, for every $m_1>0$ there exists \(N>0\)
such that, if \(i>N\),
\begin{equation}\label{24}
|\mathcal{K}g_{i}(t)-\mathcal{K}g_{0}(t)|
<
\epsilon
\qquad
\text{for every }t\in[0,m_{1}].
\end{equation}

Then, from \eqref{22} and \eqref{23}, if \(r\in[0,m_{0}]\),
\begin{equation}\label{28}
\begin{aligned}
|Tg_{i}(r)-Tg_{0}(r)|
\leq &
\Bigg|
\int_{r}^{m_{0}}
\left[
    \int_{0}^{s}
    \left(\frac{t}{s}\right)^{n-1}
    \mathcal{K}g_{i}(t)\, \din t
\right]^{\frac{1}{p-1}}
\, \din s
\\
&-
\int_{r}^{m_{0}}
\left[
    \int_{0}^{s}
    \left(\frac{t}{s}\right)^{n-1}
    \mathcal{K}g_{0}(t)\, \din t
\right]^{\frac{1}{p-1}}
\, \din s
\Bigg|
+
2\epsilon\\
\leq &
\Bigg|
\int_{r}^{m_{0}}
\frac{1}{p-1}
\left[
    \int_{0}^{s}
    \left(\frac{t}{s}\right)^{n-1}
    |\mathcal{K}g_{i}(t)-\mathcal{K}g_{0}(t)|
    \, \din t
\right]
\\
&\cdot
\left(
    E_{0}^{\frac{2-p}{p-1}}(s)
    +
    E_{i}^{\frac{2-p}{p-1}}(s)
\right)
\, \din s
\Bigg|
+
2\epsilon,
\end{aligned}
\end{equation}
Here Lemma \ref{15} was used, with
\begin{equation}\label{26}
\begin{aligned}
E_{i}(s)
:= 
\int_{0}^{s}
\left(\frac{t}{s}\right)^{n-1}
\mathcal{K}g_{i}(t)\, \din t
\leq  \int_{0}^{s}
A(t)\, \din t
\end{aligned}
\end{equation}
where the inequality follows from \eqref{17}.

Thus, by \eqref{24},
\[
|Tg_{i}(r)-Tg_{0}(r)|
\leq
\frac{m_{0}^2}{p-1}
\epsilon\,
2\left(\int_{0}^{m_{0}}A(t)\, \din t\right)^{\frac{2-p}{p-1}}
+
2\epsilon,
\qquad
r\in [0,m_0], i>N.
\]

On the other hand, by Lemma \ref{15}, \eqref{24}, and \eqref{26},
\begin{equation} \label{27} 
\begin{aligned}
|(Tg_{i})^\prime (r)-(Tg_{0})^\prime(r)|
\leq &
\Bigg|
\left[
    \int_{0}^{r}
    \left(\frac{t}{r}\right)^{n-1}
    \mathcal{K}g_{i}(t)\, \din t
\right]^{\frac{1}{p-1}}
-
\left[
    \int_{0}^{r}
    \left(\frac{t}{r}\right)^{n-1}
    \mathcal{K}g_{0}(t)\, \din t
\right]^{\frac{1}{p-1}}
\Bigg|\\
\leq &
\Bigg|
\frac{1}{p-1}
\left[
    \int_{0}^{r}
    \left(\frac{t}{r}\right)^{n-1}
    |\mathcal{K}g_{i}(t)-\mathcal{K}g_{0}(t)|
    \, \din t
\right]
\left(
    E_{0}^{\frac{2-p}{p-1}}(r)
    +
    E_{i}^{\frac{2-p}{p-1}}(r)
\right)
\Bigg|\\
\leq & {\frac{2R}{n(p-1)}\epsilon
\left(\int ^R_0 A(t)\, \din t\right)^{\frac{2-p}{p-1}}},
\qquad {0<r\leq R},\ i>N_R.
\end{aligned}
\end{equation}

At \(r=0\), \((Tg_i)'(0)=(Tg_0)'(0)=0\). Moreover, for \(r\geq m_0\), \eqref{23} gives
\[
|Tg_i(r)-Tg_0(r)|<2\epsilon.
\]

From \eqref{28} and \eqref{27}, we conclude that \(T\) is continuous.

\medskip

\noindent {\bf Part 2.} We show that \(T(Q)\) is relatively compact and that its closure is contained in \(Q\).

\medskip

\noindent {\it Step 2.1} We show that
\[
\overline{T(Q)}^{\,C^1_{\loc}}\subset Q.
\]

Let $\{g_j \}_{j=1}^\infty \subset Q$ be such that
$$
T g_j \to z \in C^{1}_{\loc} ([0,\infty))
.$$

Consequently, by \eqref{17}, Lemma \ref{16}, and \ref{4},
\begin{equation} \label{18}
Tg_j (r)
\leq
\funct _A (r)
\leq
\xi_{1} , \qquad r\geq 0 , j\geq 1
\end{equation}

% ============================================================
% Photo 5
% ============================================================

Additionally, by \ref{1} and \ref{2},
\[
\mathcal{K}g_j\geq B,
\]
where
\[
B(r)
:=
\phi_{p}
\left(
    r,\xi_{1}^{-1}\varphi_{p}(r)
\right)
\xi_{1}^{-\beta_p}.
\]

Then, by Lemma \ref{16} and \ref{3},
\begin{equation}\label{19}
\begin{aligned}
Tg_j
&\geq
\funct _B\\
&\geq
\frac{p-1}{n-p}
\left(
        \int_{0}^{{\ell}}
    t^{n-1}B(t)\, \din t
\right)^{\frac{1}{p-1}}
\varphi_{p}
\\
&\geq
\xi_{1}^{-1}\varphi_{p}.
\end{aligned}
\end{equation}

From \eqref{18} and \eqref{19},
\begin{equation*} 
\xi_{1}^{-1}\varphi_{p}(r)
\leq
Tg_j(r)
\leq
\xi_{1} .
\end{equation*}

Therefore,
\begin{equation} \label{20}
\xi_{1}^{-1}\varphi_{p}(r)
\leq
z (r)
\leq
\xi_{1},
\qquad r\geq 0.
\end{equation}

\medskip

The fact that
\[
|z'(r)|\leq \xi_{2}
\]
is a consequence of \ref{5} and \eqref{17}:
\begin{equation} \label{21}
\left|(Tg_j)'(r)\right|
=
\left|
-\left(
r^{1-n}
\int_{0}^{r}
s^{n-1}\mathcal{K}g_j(s)\, \din s
\right)^{\frac{1}{p-1}}
\right|
\leq
\xi_{2},
\qquad  {r>0},\ j\geq 1.
\end{equation}
Moreover, \((Tg_j)'(0)=0\).

% ============================================================
% Photo 6
% ============================================================

From \eqref{20} and \eqref{21}, it remains to show that
\[
\funct _{z} (0)<\infty.
\]
By \eqref{18} and Lemma \ref{16}\ref{13}, $0\leq z\leq\funct_A$ implies
\[
\funct_z(0)\leq\funct_{\funct_A}(0)<\infty,
\]
by \eqref{12}.

Therefore,
 $$
\overline{T(Q)}^{\,C^1_{\loc}}\subset Q.
 $$

\medskip

\noindent {\it Step 2.2} We show that \(\overline{T(Q)}^{\,C^1_{\loc}}\) is compact.

\medskip

Let $g\in Q$. By \ref{5} and \ref{6}, we have
\begin{equation}\label{29} 
\begin{aligned}
\left|(Tg)^{\prime \prime}(r)\right|
=&
\Bigg|
-\frac{1}{p-1}\left(
r^{1-n}
\int_{0}^{r}
s^{n-1}\mathcal{K}g(s)\, \din s
\right)^{\frac{2-p}{p-1}}\\
&\cdot \left[  (1-n) r^{-n} \int _0 ^r t^{n-1} \mathcal{K} g{(t)} \, \din t + \mathcal{K} g (r)  \right] \Bigg|\\
 \leq & {\frac{\xi_2^{2-p}}{p-1}n C\xi_1^{\beta_p}}, \qquad {r>0}
\end{aligned}
\end{equation}

{By continuity, the same bound holds at $r=0$.}

Let $(g_j)\subset Q$. Fix $R>0$. By \eqref{18}, $\{Tg_j\}$ is uniformly bounded on $[0,R]$. By \eqref{21} and \eqref{29}, $\{Tg_j \}$ and $\{(Tg_j)'\}$ are equicontinuous on $[0,R]$.

By the Arzel\`a--Ascoli theorem, there exist a subsequence and a function $v\in C^1([0,R])$ such that
\begin{gather*}
Tg_{R_j}\to v, \quad
(Tg_{R_j})'\to v'
\qquad\text{uniformly on }[0,R].
\end{gather*}

By a diagonal argument, there exist a subsequence \(\{g_{i}\}_{i\in \mathcal{I}}\) and a function $v\in C ^1([0,\infty))$ such that
\[
Tg_{i}\to v
\quad\text{in }C^1([0,m])
\qquad\forall m\in\mathbb N.
\]

Thus, \(T(Q)\) is relatively compact, and \(\overline{T(Q)}^{\,C^1_{\loc}}\) is compact.

\medskip

\noindent {\bf Part 3.} {The set \(Q\) is nonempty and convex, \(T:Q\to Q\) is continuous, and \(T(Q)\) is relatively compact.} Hence, by the Schauder--Tychonoff theorem \cite[Theorem 10.1]{zbMATH07063009}, there exists  {\(g_0\in Q\)} such that
$$
g_0=T(g_0).
$$

This concludes the proof of the proposition.\end{proof}
 %%%%%%%%%%%%%%%%%%%%%%%%%%%%%%%%%%%%%%%%%%%%%%%%%%%%%%%%%%%%%%%%%%%%%%%%%%%%%%%%%%%%%%%%%%%%%%%%%%%%%%%%%%%%%%%%%%%%%%%%%%%%%%%%%%%%%%%%%%%%%%%%%%%%%%%%%%%%%%%%%%%%%%%%%%%%%%%%%%%%%%%%%%%%%%%%%%%%%%%%%%%%%%%%%%%%%%%%%%%%%%%%%%%%%%%%%%%%%%%%%%%%%%%%%%%%%%%%%%%%%%%%%%%%%%%%%%%%%%%%%%%%%%%%%%%%%%%%%%%%%%%%%%%%%%%%%%%%%%%%%%%%%%%%%%%%%%%%%%%%%%%%%%%%%%%%%%%%%%%%%%%%%%%%

%%%%%%%%%%%%%%%%%%%%%%%%%%%%%%%%%%%%%%%%%%%%%%%%%%%%%%%%%%%%%%%%%%%%%%%%%%%%%%%%%%%%%%%%%%%%%%%%%%%%%%%%%%%%%%%%%%%%%%%%%%%%%%%%%%%%%%%%%%%%%%%%%%%%%%%%%%%%%%%%%%%%%%%%%%%%%%%%%%%%%%%%%%%%%%%%%%%%%%%%%%%%%%%%%%%%%%%%%%%%%%%%%%%%%%%%%%%%%%%%%%%%%%%%%%%%%%%%%%%%%%%%%%%%%%%%%%%%%%%%%%%%%%%%%%%%%%%%%%%%%%%%%%%%%%%%%%%%%%%%%%%%%%%%%%%%%%%%%%%%%%
%%%%%%%%%%%%%%%%%%%%%%%%%%

\medskip

\noindent {\bf Example for $F_p$.}   For each integer $n\geq2$, set
\begin{equation*}\label{a6}
c_n:=\int_0^1\frac{t^{n-1}}{(1+t)^{2n}}\,dt
=\frac{((n-1)!)^2}{2(2n-1)!},\qquad
\ell:=\left[\frac{8(n-1)}{c_n}\right]^{2/(n-1)},
\end{equation*}
and choose
\begin{equation*}\label{a7}
\xi_1:=\ell,\qquad\xi_2:=1,\qquad\beta_p:=\frac{p-1}{2}.
\end{equation*}
Define, for $(r,u,\rho)\in[0,\infty)^3$,
\begin{equation*}\label{a8}
F_p(r,u,\rho)=F_1(r,u,\rho)
:=\frac{\left(1+\dfrac{u}{1+u}\right)
\left(1+\dfrac{\rho^2}{1+\rho^2}\right)}
{4\ell^{(n+1)/2}(1+r/\ell)^{2n}},
\end{equation*}
and
\begin{equation*}\label{a9}
\phi_p(r,u)=\phi_1(r,u)=F_1(r,u,0).
\end{equation*}

Then  \ref{1}--\ref{7} hold for every $1<p<\min\{2,\sqrt n\}$ and every $0<\alpha\leq1$. The constant in \ref{6} may be chosen as $C=1$, independently of $p$.

%%
%%
%%
%%
%%
%%

%%
%%
%%
%%
%%%
%%
%%
%%%

%%
%%
%%
%%

%%
%%
%%
%%
%%
%%

%%
%%
%%
%%
%%%
%%
%%
%%%

%%
%%
%%
%%

%%
%%
%%
%%

\section{The limit $p\to 1$}\label{43}

This section establishes the results on subsequence convergence and identification of the limits needed to prove Theorems \ref{80} and \ref{81}. We first obtain uniform estimates and extract jointly convergent subsequences, then analyze the sets where $\mathtt G<1$ and $\mathtt G=1$, and finally complete the proofs of the main theorems.

%The limit $p\to1$ and related $1$-Laplacian problems have been studied by Mercaldo--Segura de Le\'on--Trombetti \cite{MercaldoSeguraTrombetti2008}, Molino--Segura de Le\'on \cite{MolinoSeguraLeon2022}, Della Pietra--Oliva--Segura de Le\'on \cite{DellaPietraOlivaSegura2024}, and Balducci and collaborators \cite{BalducciOlivaPetitta2024,Balducci2025}.

\subsection{Uniform estimates and subsequential convergence}  Let
\[
u_p(x)=v_p(|x|)
\]
be the positive radial solution supplied by Proposition \ref{25}. Then
\begin{equation}\label{34}
\xi_1^{-1}\ftes _p   (|x|)
\leq u_p(x)\leq\xi_1,
\qquad
|\nabla u_p(x)|\leq\xi_2,
\qquad
v_p'\leq0
\end{equation}
where $\ftes _p   (r)=\max\{\ell,r\}^{-\frac{n-p}{p-1}}$.

Since \eqref{e6} holds for all $p>1$ sufficiently close to $1$, with
$\ell$ and $\xi_1$ independent of $p$, it follows that
\begin{equation*}\label{eq:ell-ge-one-43}
\ell\geq1.
\end{equation*}

Define, for $r\geq0$, \label{115}
\begin{equation*}\label{eq:Psi-A-z-43}
\mathtt{K}_p (r)
:= \mathcal K v_p (r) = F_p\!\left(r,v_p(r),|v_p'(r)|\right)v_p(r)^{-\beta _p },
\end{equation*}
\begin{equation*}\label{eq:Ap-def-43}
\mathtt{G}_p (r)
:= \left[-(T v _p )^\prime\right] ^{p-1} (r) =r^{1-n}\int_0^r s^{n-1}\mathtt{K}_p (s)\,\din s
\quad(r>0),
\qquad \mathtt{G}_p (0):=0,
\end{equation*}
and
\begin{equation*}\label{eq:zp-def-43}
{\bf z}_p
:=|\nabla u_p|^{p-2}\nabla u_p.
\end{equation*}

Then
\begin{equation}\label{eq:radial-identities-43}
-v_p'(r)=\mathtt{G}_p (r)^{\frac1{p-1}},
\qquad
\mathtt{G}_p (r)=|v_p'(r)|^{p-1},
\qquad
{\bf z}_p(x)=-\mathtt{G}_p (|x|)\frac{x}{|x|},
\end{equation}
for $r>0$ and $x\neq0$. Moreover,
\begin{equation}\label{eq:Ap-ode-43}
\mathtt{G}_p '(r)+\frac{n-1}{r}\mathtt{G}_p (r)=\mathtt{K}_p (r),
\qquad r>0.
\end{equation}

\begin{lemma}\label{30}
For every $R>0$ there exist $p_R>1$ and $0\leq C_R<\infty$, independent
of $p\in(1,p_R)$, such that
\begin{enumerate}[label = $(\roman*)$]
\item \label{47} \begin{equation}\label{eq:Psi-uniform-43}
\mathtt{K}_p (r)\leq C_R,
\qquad0\leq r\leq R,
\end{equation}

\item \label{48} \begin{equation}\label{eq:Ap-basic-bounds-43}
 \mathtt{G}_p (r)
\leq\min\left\{\frac{C_R}{n}r,\xi_2^{p-1}\right\},
\qquad0\leq r\leq R,
\end{equation}
and
\begin{equation}\label{eq:Ap-derivative-bound-43}
|\mathtt{G}_p '(r)|\leq\frac{2n-1}{n}C _R
\qquad\text{for a.e. }r\in(0,R).
\end{equation}

\item \label{49} 
\begin{equation}\label{42}
\left|
\partial_{x_i}
\left(
|\nabla u_p|^{p-2}(u_p)_{x_k}
\right)
\right|
\leq2 C_R
\qquad\text{a.e. in }B_R,
\end{equation}
and
\begin{equation}\label{eq:zp-W1infty-43}
\sup_{1<p<p_R}
\|{\bf z}_p\|_{W^{1,\infty}(B_R;\mathbb R^n)}<\infty.
\end{equation}
\end{enumerate}
\end{lemma}

\begin{proof}
\noindent \ref{47} By \eqref{33}, there are $p_R>1$ and $A_R<\infty$ such that
\[
0\leq F_p(r,t,\rho)\leq A_R
\]
for $(r,t,\rho)\in[0,R]\times[0,\xi_1]\times[0,\xi_2]$ and $1<p<p_R$. Moreover, \eqref{34} gives
\[
u_p(r)^{-\beta_p}
\leq \xi_1^{\beta_p} \max\{\ell,R\} ^{(n-p)\beta_p/(p-1)}.
\]

Since \eqref{32} implies that $\beta_p/(p-1)$ is bounded for $p$ close to $1$, the right-hand side is bounded independently of $p$. Hence, after decreasing $p_R$ if necessary,
\[
C_R:=\sup_{\substack{0\leq r\leq R\\1<p<p_R}}\mathtt{K}_p(r)<\infty,
\]
which proves \eqref{eq:Psi-uniform-43}.

\medskip 

\noindent \ref{48} We have
\[
 \mathtt{G}_p (r)
\leq r^{1-n}C_R\int_0^r s^{n-1}\, \din s
=\frac{C_R}{n}r.
\]
The second bound in \eqref{eq:Ap-basic-bounds-43} follows from
\eqref{eq:radial-identities-43} and $|v_p'|\leq\xi_2$.  

From \eqref{eq:Ap-ode-43},
\[
|\mathtt{G}_p '(r)|
\leq{C_R}+\frac{n-1}{r}\mathtt{G}_p (r)
\leq{\frac{2n-1}{n}C_R},
\]
which proves \eqref{eq:Ap-derivative-bound-43}.

\medskip

\noindent \ref{49} For $r=|x|>0$,
\[
|\nabla u_p|^{p-2}(u_p)_{x_k}
=-\mathtt{G}_p (r)\frac{x_k}{r}.
\]
Differentiation gives
\begin{equation*}\label{eq:flux-derivative-43}
\begin{aligned}
\partial_{x_i}
\left(
|\nabla u_p|^{p-2}(u_p)_{x_k}
\right)
&=-\mathtt{G}_p '(|x|)\frac{x_ix_k}{|x|^2}
\\
&\quad
-\mathtt{G}_p (|x|)
\left(
\frac{\delta_{ik}}{|x|}
-\frac{x_ix_k}{|x|^3}
\right).
\end{aligned}
\end{equation*}

Therefore,
\[
\left|
\partial_{x_i}\left(|\nabla u_p|^{p-2}(u_p)_{x_k}\right)
\right|
\leq |\mathtt G_p'(|x|)|+\frac{\mathtt G_p(|x|)}{|x|}
\leq\frac{2n-1}{n}C_R+\frac1nC_R
=2C_R,
\]
which proves \eqref{42}.

Finally,
\[
|{\bf z}_p(x)|=\mathtt{G}_p (|x|)\leq\frac{C_R}{n}|x|,
\]
hence ${\bf z}_p(0):=0$ gives a Lipschitz extension to $B_R$ and
\eqref{eq:zp-W1infty-43} follows.
\end{proof}

\begin{proposition}\label{44}
There exist a subsequence $\{u_{p_j}\}$, a radial nonincreasing function
\[
u\in W^{1,\infty}(\mathbb R^n),
\qquad
u(x)=v(|x|),
\qquad
0\leq u\leq\xi_1,
\qquad
\|\nabla u\|_{L^\infty(\mathbb R^n)}\leq\xi_2,
\]
a radial measurable function
\[
w:\mathbb R^n\to(0,1],
\]
a function
\[
\mathtt{G}\in W^{1,\infty}_{\mathrm{loc}}([0,\infty)),
\qquad
0\leq\mathtt{G}\leq1,
\qquad
\mathtt{G}(0)=0,
\]
a vector field
\[
{\bf z}\in W^{1,\infty}_{\mathrm{loc}}(\mathbb R^n;\mathbb R^n),
\qquad
|{\bf z}|\leq1,
\]
and a function
\[
\mathtt{K}\in L^\infty_{\mathrm{loc}}([0,\infty)),
\qquad {\mathtt K\geq0\quad\text{a.e.}},
\]
such that
\begin{enumerate}[label=$(\roman*)$]

\item \label{50}
\begin{equation}\label{35}
u_{p_j}\to    u
\quad\text{locally uniformly in }\mathbb R^n.
\end{equation}

Moreover,
\begin{equation}\label{eq:grad-weak-star-43}
\nabla u_{p_j}\stackrel{*}{\rightharpoonup}\nabla u
\quad\text{in }L^\infty_{\mathrm{loc}}(\mathbb R^n;\mathbb R^n) .
\end{equation}

\item \label{51} $|v_{p_j}'|\stackrel{*}{\rightharpoonup}|v'|
$ in $L^\infty_{\mathrm{loc}}
\bigl((0,\infty);r^{n-1}\din r\bigr)$ and $ |\nabla u_{p_j}|\stackrel{*}{\rightharpoonup}|\nabla u|$ in $L^\infty_{\mathrm{loc}}(\mathbb R^n)$.

\item \label{52} Setting $w_j:=u_{p_j}^{p_j-1}$ and $c_j:=\frac{\beta_{p_j}}{p_j-1}$, we have
\begin{gather*}
w_j\to    w
\quad\text{a.e. in } \mathbb R^n,\\
w_j\to    w
\quad\text{ in } L^q_{\mathrm{loc}}(\mathbb R^n),
\qquad
1\leq q<\infty,\\
\bigl(\max\{\ell,|x|\}\bigr)^{-(n-1)}
\leq w(x)\leq1
\qquad\text{for a.e. }x\in\mathbb R^n,\\
w=1\qquad\text{a.e. on }\{x\in\R^n\mid u(x)>0\}
\end{gather*}
and
\begin{equation}\label{36}
w_j^{-c_j}
\to    w^{-{\lbp}}
\quad\text{in }L^q_{\mathrm{loc}}(\mathbb R^n),
\qquad
1\leq q<\infty.
\end{equation}

\item \label{53}
\begin{equation}\label{37}
{\bf z}_{p_j}\to   {\bf z}
\quad\text{locally uniformly in }\mathbb R^n,
\qquad
\nabla{\bf z}_{p_j}\stackrel{*}{\rightharpoonup}\nabla{\bf z}
\quad\text{in }L^\infty_{\mathrm{loc}}(\mathbb R^n;\mathbb R^{n\times n}),
\end{equation}
\begin{equation}\label{eq:Ap-limit-43}
\mathtt{G}_{p_j}\to   \mathtt{G}
\quad\text{locally uniformly on }[0,\infty),
\qquad
{\bf z}(x)=-\mathtt{G}(|x|)\frac{x}{|x|}
\quad(x\neq0),
\end{equation}

\item \label{56}
 \begin{equation}\label{94}
\mathtt{K}_{p_j}(|x|)\stackrel{*}{\rightharpoonup}
\mathtt{K}(|x|)
\quad\text{in }L^\infty_{\mathrm{loc}}(\mathbb R^n).
\end{equation}

Moreover, 
\begin{equation}\label{92}
-\operatorname{div}{\bf z}
=\mathtt{K}(|x|),
\qquad
\mathtt{K}(r)
=\mathtt{G}'(r)+\frac{n-1}{r}\mathtt{G}(r)
\quad\text{for a.e. }r>0.
\end{equation}

\item \label{57}
\begin{gather}
\label{59}
{\bf z}\cdot\nabla u=|\nabla u|
\quad\text{a.e. in }\mathbb R^n,
\qquad
(1-\mathtt{G}(r))|v'|(r)=0
\quad\text{for a.e. }r>0,\\
\label{eq:calibration-43}
({\bf z},Du)=|Du|.
\end{gather}

Consequently,
\begin{equation}\label{60}
|v'|
=|v'|\,{\bf 1}_{\{\mathtt{G}=1\}}
\quad\text{a.e. on }(0,\infty),
\qquad
v(r)=v(0)-\int_0^r
|v'|(s){\bf 1}_{\{\mathtt{G}(s)=1\}}\,\din s.
\end{equation}
\end{enumerate}
\end{proposition}

\begin{proof}
%Throughout the proof, subsequences obtained by successive extractions are not relabeled.

\noindent\ref{50}
From \eqref{34},
\[
|u_p(x)-u_p(y)|\leq\xi_2|x-y|,
\qquad0<u_p\leq\xi_1.
\]
Arzel\`a--Ascoli and a diagonal argument give \eqref{35}. Passing to the
limit in the Lipschitz inequality yields
\[
|u(x)-u(y)|\leq\xi_2|x-y|,
\qquad u(x)=v(|x|),
\]
so
\[
u\in W^{1,\infty}(\mathbb R^n),
\qquad D^su=0.
\]

By the Banach--Alaoglu theorem \cite{zbMATH05633610}, after passing to a subsequence, one has
$$
\nabla u_{p_j}\stackrel{*}{\rightharpoonup}\overline{{\bf u}}
\quad \text{in }L^\infty_{\mathrm{loc}}(\mathbb R^n;\mathbb R^n).
$$
Then, for every $\Phi\in C_c^\infty(\mathbb R^n;\mathbb R^n)$,
\[
\int_{\mathbb R^n}\nabla u_{p_j}\cdot\Phi\,\din x
=-\int_{\mathbb R^n}u_{p_j}\operatorname{div}\Phi\,\din x
\to    \int_{\mathbb R^n} \overline{{\bf u}}\cdot\Phi\,\din x =
-\int_{\mathbb R^n}u\operatorname{div}\Phi\,\din x.
\]
Hence $\nabla u=\overline{{\bf u}}$, proving \eqref{eq:grad-weak-star-43}.

%\fragmento{frag1}

\medskip

\noindent\ref{51} Since $v_{p_j}'\leq0$ and $v'\leq0$, the first convergence follows from
\[
\begin{aligned}
\int_0^\infty
|v_{p_j}'(r)|\zeta (r)r^{n-1}\,\din r
&=
-\int_0^\infty
v_{p_j}'(r) \zeta (r)r^{n-1}\,\din r
\\
&=
\int_0^\infty
v_{p_j}(r)
\frac{\din}{\din r}
\bigl(\zeta(r)r^{n-1}\bigr)\,\din r
\end{aligned}
\]
for $\zeta\in C_c^\infty((0,\infty))$. The local uniform convergence $v_{p_j}\to v$ gives
$$
\int_0^\infty
|v_{p_j}'(r)|\zeta (r)r^{n-1}\,\din r \to  \int_0^\infty
|v'(r)|\zeta (r)r^{n-1}\,\din r.
$$
For every $R>0$, the density of $C_c^\infty((0,R))$ in $L^1((0,R);r^{n-1}\din r)$ and the bound $|v_{p_j}'|\leq\xi_2$ extend this convergence to all tests in that weighted $L^1$ space, including tests whose support meets $r=0$.

For the second convergence, polar coordinates give
\[
\begin{aligned}
\int_{B_R}
|\nabla u_{p_j}(x)|\phi(x)\,\din x
&=
\int_0^R
|v_{p_j}'(r)|\Phi(r)r^{n-1}\,\din r
\end{aligned}
\]
where $\phi\in L^1(B_R)$ and $\Phi(r)=\int_{\mathbb S^{n-1}}\phi(r\theta)\,\din\sigma(\theta)$.
Since $\Phi\in L^1((0,R);r^{n-1}\din r)$, the first convergence applies and proves \ref{51}.

\medskip

\noindent\ref{52} From \eqref{34},
\begin{equation}\label{55}
\xi_1^{-(p_j-1)}
\bigl(\max\{\ell,r\}\bigr)^{-(n-p_j)}
\leq
\widetilde w_j (r)
\leq \xi_1^{p_j-1} \qquad \widetilde w_j = v_{p_j} ^{p_j-1} .
\end{equation}

The functions $\widetilde w_j$ are nonincreasing and locally uniformly bounded. By Helly's selection theorem \cite[Chap.~8, Sect.~4, Lemma~2]{natanson1961theory}, there is a function $\widetilde w:[0,\infty)\to(0,\infty)$
such that, up to a subsequence,
\[
\widetilde w_j(r) \to    \widetilde w (r)
\]
for every $r\geq 0$. 

Set
$$
w(x):=\widetilde w (|x|).
$$

Passing to the limit in \eqref{55} gives
\[
\bigl(\max\{\ell,r\}\bigr)^{-(n-1)}
\leq\widetilde w(r)\leq1.
\]
If $u(x)>0$, local uniform convergence implies that $u_{p_j}(x)$ stays bounded away from $0$; hence
\[
\log w_j(x)=(p_j-1)\log u_{p_j}(x)\to   0,
\qquad w(x)=1.
\]

On every ball, \eqref{55} bounds $w_j$ away from $0$, while $c_j\to{\lbp}$. Consequently, $w_j^{-c_j}\to w^{-{\lbp}}$ pointwise, and both sequences are uniformly bounded there.

For $1\leq q<\infty$, dominated convergence gives
\begin{gather*}
\|w_j-w\|_{L^q(B_R)}\to   0,\\
\|w_j^{-c_j}-w^{-{\lbp}}\|_{L^q(B_R)}
\to   0.
\end{gather*}

This proves \ref{52}.

\medskip

\noindent\ref{53} By Lemma \ref{30},
\[
\mathtt{G}_p(r)
\leq\min\left\{\frac{C_R}{n}r,\xi_2^{p-1}\right\},
\qquad
0\leq r\leq R,
\]
and
\[
\|\mathtt{G}_p'\|_{L^\infty(0,R)}
\leq\frac{2n-1}{n}C_R.
\]
Hence $\{\mathtt{G}_{p_j}\}$ is uniformly bounded and equi-Lipschitz on
$[0,R]$. Arzel\`a--Ascoli and a diagonal extraction yield
\[
\mathtt{G}_{p_j}\to   \mathtt{G}
\quad\text{uniformly on }[0,R]
\quad\text{for every }R>0,
\]
where
\[
\mathtt{G}\in W^{1,\infty}_{\mathrm{loc}}([0,\infty))
, \qquad \mathtt{G}(0)=0, \qquad
0\leq\mathtt{G}\leq1.
\]

For $x\neq0$, \eqref{eq:radial-identities-43} gives
\[
{\bf z}_{p_j}(x)
=
-\mathtt{G}_{p_j}(|x|)\frac{x}{|x|}.
\]

Set
\[
{\bf z}(x):=
-\mathtt{G}(|x|)\frac{x}{|x|},
\qquad {\bf z}(0):=0.
\]

Then
\[
\sup_{x\in B_R}
|{\bf z}_{p_j}(x)-{\bf z}(x)|
\leq
\sup_{0\leq r\leq R}
|\mathtt{G}_{p_j}(r)-\mathtt{G}(r)|
\to   0.
\]
Thus
\[
{\bf z}_{p_j}\to   {\bf z}
\quad\text{uniformly on }B_R.
\]
Moreover, Lemma \ref{30} gives
\[
\sup_j
\|{\bf z}_{p_j}\|_{W^{1,\infty}(B_R;\mathbb R^n)}
<\infty.
\]
By the Banach--Alaoglu theorem \cite{zbMATH05633610},
\[
\nabla{\bf z}_{p_j}
\stackrel{*}{\rightharpoonup}Z_R
\quad\text{in }
L^\infty(B_R;\mathbb R^{n\times n})
\]
along a subsequence. For every
$\Phi\in C_c^\infty(B_R;\mathbb R^{n\times n})$,
the uniform convergence ${\bf z}_{p_j}\to{\bf z}$ implies
\[
\int_{B_R}Z_R:\Phi\,\din x
=
-\int_{B_R}{\bf z}\cdot\operatorname{div}\Phi\,\din x.
\]
Therefore
\[
Z_R=\nabla{\bf z},
\]
so that
\[
{\bf z}\in W^{1,\infty}(B_R;\mathbb R^n)
\]
and \eqref{37} follows. Also,
\[
|{\bf z}(x)|=\mathtt{G}(|x|)\leq1.
\]

\medskip

\noindent \ref{56} By \eqref{eq:Psi-uniform-43},
\[
\|\mathtt{K}_{p_j}\|_{L^\infty(0,R)}
\leq C_R.
\]
Banach--Alaoglu and a diagonal extraction give
\[
\mathtt{K}_{p_j}
\stackrel{*}{\rightharpoonup}
\mathtt{K}
\quad\text{in }L^\infty(0,R)
\quad\text{for every }R>0,
\]
for some
\[
\mathtt{K}\in L^\infty_{\mathrm{loc}}([0,\infty)),
\qquad {\mathtt K\geq0\quad\text{a.e.}}.
\]

\eqref{94} follows from polar coordinates:
\[
\int_{B_R}
\mathtt{K}_{p_j}(|x|)\phi(x)\,\din x
=
\int_0^R
\mathtt{K}_{p_j}(r)\Phi(r)r^{n-1}\,\din r
\]
where $\phi\in L^1(B_R)$ and $\Phi(r):=\int_{\mathbb S^{n-1}}\phi(r\theta)\,\din\sigma(\theta)$.

\medskip

We now prove \eqref{92}.

\medskip

For $x\neq0$, using \eqref{eq:Ap-ode-43},
\[
-\operatorname{div}{\bf z}_{p_j}(x)
=
\mathtt{G}_{p_j}'(|x|)
+\frac{n-1}{|x|}\mathtt{G}_{p_j}(|x|)
=
\mathtt{K}_{p_j}(|x|)
\]
for a.e. $x\neq0$. 

Since ${\bf z}_{p_j}\in W^{1,\infty}_{\mathrm{loc}}(\mathbb R^n;\mathbb R^n)$, for every
$\varphi\in C_c^\infty(\mathbb R^n)$,
\[
\int_{\mathbb R^n}
{\bf z}_{p_j}\cdot\nabla\varphi\,\din x
=
\int_{\mathbb R^n}
\mathtt{K}_{p_j}(|x|)\varphi\,\din x.
\]

Passing to the limit by \eqref{37} and \eqref{94},
\[
\int_{\mathbb R^n}
{\bf z}\cdot\nabla\varphi\,\din x
=
\int_{\mathbb R^n}
\mathtt{K}(|x|)\varphi\,\din x.
\]

Hence
\[
-\operatorname{div}{\bf z}
=
\mathtt{K}(|x|).
\]

On the other hand, for a.e. $x\neq0$,
\[
{\bf z}(x)
=
-\mathtt{G}(|x|)\frac{x}{|x|}
\]
and therefore
\[
-\operatorname{div}{\bf z}(x)
=
\mathtt{G}'(|x|)
+\frac{n-1}{|x|}\mathtt{G}(|x|).
\]
Consequently,
\[
\mathtt{K}(r)
=
\mathtt{G}'(r)+\frac{n-1}{r}\mathtt{G}(r)
\quad\text{for a.e. }r>0,
\]
which proves \eqref{92}.

\medskip

\noindent\ref{57} We prove \eqref{59}. By \eqref{34} and \eqref{eq:radial-identities-43},
\[
 |v_{p_j}'|\leq\xi_2,
\qquad
\mathtt{G}_{p_j}(r)=|v_{p_j}'(r)|^{p_j-1}.
\]

Hence
\[
\mathtt{G}_{p_j}|v_{p_j}'|=|v_{p_j}'| ^{p_j}.
\]

Set
\[
\delta_j:=
\sup_{0\leq t\leq\xi_2}|t^{p_j}-t|
\to   0.
\]
Therefore
\begin{equation*}
\bigl|
\mathtt{G}_{p_j}(r)|v_{p_j}'(r)|-|v_{p_j}'(r)|
\bigr|
\leq\delta_j
\quad\text{for a.e. }r>0.
\end{equation*}

Fix $R>0$ and let
\[
\psi\in L^1((0,R);r^{n-1}\din r).
\]
Then
\[
\begin{aligned}
&\left|
\int_0^R
\psi
\bigl(
\mathtt{G}_{p_j}|v_{p_j}'| -|v_{p_j}'|
\bigr)
r^{n-1}\,\din r
\right|
\\
&\qquad\leq
\delta_j
\int_0^R|\psi|r^{n-1}\,\din r
\to   0.
\end{aligned}
\]

Moreover,
\[
\begin{aligned}
&\int_0^R
\psi\,\mathtt{G}_{p_j}|v_{p_j}'|\,r^{n-1}\,\din r
-
\int_0^R
\psi\,\mathtt{G}|v'|\,r^{n-1}\,\din r
\\
&=
\int_0^R
\psi(\mathtt{G}_{p_j}-\mathtt{G})|v_{p_j}'|  \,r^{n-1}\,\din r
+
\int_0^R
\psi\mathtt{G}(|v_{p_j}'|-|v'|)\,r^{n-1}\,\din r.
\end{aligned}
\]
The first term tends to $0$ because $\mathtt{G}_{p_j}\to   \mathtt{G}
\quad\text{uniformly on }[0,R]$, $ |v_{p_j}'|\leq\xi_2$, and the second tends to $0$ by \ref{51}. 

Hence
\[
\int_0^R
\psi(r)\mathtt{G}(r)|v'(r)|r^{n-1}\,\din r
=
\int_0^R
\psi(r)|v'(r)|r^{n-1}\,\din r.
\]
Since $\psi$ is arbitrary,
\[
(1-\mathtt{G}(r))|v'(r)|=0
\quad\text{for a.e. }r\in(0,R).
\]

Furthermore, for a.e. $x\neq0$,
\[
\nabla u(x)
=
v'(|x|)\frac{x}{|x|}
=
-|v'(|x|)|\frac{x}{|x|}, \qquad
{\bf z}(x)
=
-\mathtt{G}(|x|)\frac{x}{|x|}.
\]
Therefore
\[
{\bf z}(x)\cdot\nabla u(x)
=
\mathtt{G}(|x|)|v'(|x|)|
=
|v'(|x|)|
=
|\nabla u(x)|
\]
for a.e. $x\in\mathbb R^n$. This proves
\eqref{59}.

\medskip

We now prove
\begin{equation}\label{58}
({\bf z},Du)=|Du|.
\end{equation}
Since $u\in W^{1,\infty}_{\mathrm{loc}}(\mathbb R^n)$ and ${\bf z}\in W^{1,\infty}_{\mathrm{loc}}(\mathbb R^n;\mathbb R^n)$, for every $\varphi\in C_c^\infty(\mathbb R^n)$,
\begin{equation}\label{93}
\begin{aligned}
\langle({\bf z},Du),\varphi\rangle
&=-\int_{\mathbb R^n}u\varphi\,\operatorname{div}{\bf z}\,\din x
-\int_{\mathbb R^n}u{\bf z}\cdot\nabla\varphi\,\din x
\\
&=\int_{\mathbb R^n}\varphi\,{\bf z}\cdot\nabla u\,\din x
=\int_{\mathbb R^n}\varphi\,|\nabla u|\,\din x.
\end{aligned}
\end{equation}

By \eqref{122} gives
\[
|Du|=\mathcal L^n \llcorner |\nabla u|.
\]
Thus $({\bf z},Du)=|Du|$, which proves \eqref{eq:calibration-43}.

\medskip

Finally, because
\[
\mathtt{G}\leq1,
\qquad
(1-\mathtt{G})|v'|=0,
\]
we have
\[
|v'|
=
|v'|\,{\bf 1}_{\{\mathtt{G}=1\}}
\quad\text{a.e. on }(0,\infty).
\]
Since $v\in W^{1,\infty}_{\mathrm{loc}}([0,\infty))$ and $v'\leq0$,
\[
\begin{aligned}
v(r)
&=
v(0)+\int_0^r v'(s)\,\din s
\\
&=
v(0)-\int_0^r
|v'|(s){\bf 1}_{\{\mathtt{G}(s)=1\}}\,\din s.
\end{aligned}
\]
This proves \eqref{60}.
\end{proof}

\subsection{Behavior of the limit and source identification}
We first show that $v$ is constant on every interval contained in $\{r>0 \mid \mathtt G(r)<1\}$. Convexity or concavity of $\rho\mapsto F_1(r,t,\rho)$ gives bounds for $\mathtt K$, while affine dependence on $\rho$ allows us to identify $\mathtt K$ in terms of $F_1$, $u$, and $|\nabla u|$. On measurable subsets of $\{r>0 \mid \mathtt G(r)=1\}$, the assumption $(\mathtt G_{p_j}(r)-1)/(p_j-1)\to a(r)\in\mathbb R$ a.e. gives $|v'(r)|=e^{a(r)}$ a.e. and identifies $\mathtt K$. Finally, under the corresponding assumptions, comparison estimates and the uniform bound for $|v_{p_j}'|$ give bounds for the limit inferior and limit superior of $(\mathtt G_{p_j}(r)-1)/(p_j-1)$.

In the following result, the relation $|\nabla u_{p_j}|=\mathtt G_{p_j}^{1/(p_j-1)}$ allows us to split the analysis of the limit $\mathtt K$ into the regions ${\mathtt G<1}$ and ${\mathtt G=1}$.
\begin{lemma}\label{82}
Set
\begin{equation*}\label{eq:subcritical-critical-sets-43}
\widehat{\mathscr{C}}_<:=\{x\in\mathbb R^n \mid \mathtt{G}(|x|)<1\},
\qquad
\widehat{\mathscr{C}}_1:=\{x\in\mathbb R^n\setminus\{0\}\mid \mathtt{G} (|x|)=1\}.
\end{equation*}

\begin{enumerate}[label=$(\roman*)$]
\item \label{62} For every $1\leq q<\infty$,
\begin{equation}\label{eq:gradient-zero-subcritical-43}
\begin{gathered}
|\nabla u_{p_j}|\to   0
\quad\text{in }L^q_{\mathrm{loc}}(\widehat{\mathscr{C}}_<),\\
{|\nabla u_{p_j}(x)|\to   0\quad\text{for every }x\in\widehat{\mathscr{C}}_<},
\qquad \nabla u=0\quad\text{a.e. in }\widehat{\mathscr{C}}_<,
\end{gathered}
\end{equation}
\begin{equation}\label{eq:Psi-subcritical-43}
\mathtt{K}(|x|)
=w^{-{\lbp}}F_1(|x|,u(x),0)
\qquad\text{a.e. }x\in \widehat{\mathscr{C}}_<,
\end{equation}
and, if $|\widehat{\mathscr{C}}_1|\neq 0$,
\begin{equation}\label{eq:Psi-critical-43}
\mathtt{K}(|x|)=\frac{n-1}{|x|}
\qquad\text{a.e. }x\in\widehat{\mathscr{C}}_1.
\end{equation}

\item \label{63} There exist $r_0>0$ and $\delta\in(0,1)$ such that, for all sufficiently large $j$,
\begin{equation}\label{eq:origin-degeneration-43}
\sup_{|x|\leq r_0}|\nabla u_{p_j}(x)|
\leq \delta^{1/(p_j-1)}\to   0.
\end{equation}

Moreover, if $|\{|\nabla u|>0\}|\neq 0$,
\begin{equation}\label{eq:positive-gradient-critical-43}
\{|\nabla u|>0\}\subset\widehat{\mathscr{C}}_1\quad\text{up to a null set},
\qquad
{\mathtt{K}(|x|)=\frac{n-1}{|x|}}\quad\text{a.e. on }\{|\nabla u|>0\}.
\end{equation}
\end{enumerate}
\end{lemma}
\begin{proof}

 \noindent \ref{62} If $x\in\widehat{\mathscr{C}}_<$, then
\[
\mathtt{G}_{p_j}(|x|)\to    \mathtt{G}(|x|)<1,
\]
and therefore
\[
|\nabla u_{p_j}(x)|
=\mathtt{G}_{p_j}(|x|)^{1/(p_j-1)}\to   0.
\]

The bound $|\nabla u_{p_j}|\leq\xi_2$ and dominated convergence give the strong local $L^q$
convergence in \eqref{eq:gradient-zero-subcritical-43}. The second identity in \eqref{59} gives $\nabla u=0$ a.e. in $\widehat{\mathscr{C}}_<$. 

By \eqref{33}, \eqref{35}, \eqref{36}, and dominated
convergence,
\[
\mathtt{K} _{p_j}(|x|)
\to   
w^{-{\lbp}}F_1(|x|,u(x),0)
\quad\text{in }L^1_{\mathrm{loc}}(\widehat{\mathscr{C}}_<),
\]
which proves \eqref{eq:Psi-subcritical-43}.

Since $\mathtt{G}'=0$ for a.e. $r\in\{\mathtt{G}=1\}$, \eqref{92} gives
\eqref{eq:Psi-critical-43}.

\medskip

\noindent \ref{63} Choose $R=1$ in \eqref{eq:Psi-uniform-43} and let $j_0>0$ such that $p_j \in (1,p_R)$ for $j>j_0$. If $C_1=0$,
then $\mathtt{G}_{p_j}\equiv0$ on $[0,1]$; in this case, take $r_0=\delta=1/2$. If
$C_1>0$, choose
\[
0<r_0<\min\left\{1,\frac{n}{C_1}\right\},
\qquad
\delta:=\frac{C_1r_0}{n}<1.
\]

Then \eqref{eq:Ap-basic-bounds-43} gives
\[
|\nabla u_{p_j}(x)|
=\mathtt{G} _{p_j}(|x|)^{1/(p_j-1)}
\leq \delta^{1/(p_j-1)},
\qquad|x|\leq r_0,
\] 
which proves \eqref{eq:origin-degeneration-43}.

Passing to the limit in $\mathtt{G}_{p_j}(r)\le \delta$ for $0\le r\le r_0$ gives
\[
\mathtt{G}(r)\le \delta<1,
\]
hence $B_{r_0}\subset\widehat{\mathscr{C}}_<$. 

Finally,
\eqref{59} implies
$\{|\nabla u|>0\}\subset\widehat{\mathscr{C}}_1$ up to a null set, and
\eqref{eq:Psi-critical-43} gives the second assertion in
\eqref{eq:positive-gradient-critical-43}.\end{proof}

%\begin{remark}\label{rem:mercado-comparison-43} In \cite{MercaldoSeguraTrombetti2008}, the radial flux in the model considered there is independent of $p$.  Hence on its critical set $|{\bf z}|=1$ one has exactly \[ |\nabla u_p|=1,\] whereas $|\nabla u_p|\to0$ on $\{|{\bf z}|<1\}$; this yields an explicit limit profile supported by the unit-flux set.  In the present problem $\mathtt{G}_p $ depends on $p$ through $u_p$ and through the gradient dependence of $F_p$.  Therefore $A(r)=1$ alone does not determine the limit of $|\nabla u_p(r)|$. \end{remark}

%%%
%%%
%%%

The proof uses Mazur's lemma  to form convex combinations of $|\nabla u_{p_j}|$ that converge strongly to $|\nabla u|$ in $L^2(B_R)$. Convexity and dominated convergence then give the bound for the limiting source. Concavity gives the reverse inequality, and affine dependence gives equality, identifying the nonlinear term in the limit equation.
\begin{proposition}\label{64}
Under the notation and conclusions of Proposition \ref{44}, assume that, for every $(r,t)\in[0,\infty)\times[0,\xi_1]$, the map $\rho\mapsto   F_1(r,t,\rho)$ is convex on $[0,\xi_2]$. Then
\begin{equation}\label{eq:convex-bound-43}
F_1(|x|,u,|\nabla u|)
\leq w^{\lbp}{\mathtt{K}(|x|)}
\qquad\text{a.e. in }\mathbb R^n.
\end{equation}

If, instead, the same map is concave on $[0,\xi_2]$, then
\begin{equation}\label{eq:concave-bound-43}
w^{\lbp}{\mathtt{K}(|x|)}
\leq F_1(|x|,u,|\nabla u|)
\qquad\text{a.e. in }\mathbb R^n.
\end{equation}

In particular, if it is affine in $\rho$ on $[0,\xi_2]$, then
\begin{equation}\label{eq:affine-identification-43}
w^{\lbp}{\mathtt{K}(|x|)}
=F_1(|x|,u,|\nabla u|)
\qquad\text{a.e. in }\mathbb R^n,
\end{equation}
and therefore
\begin{equation*}\label{eq:affine-limit-problem-43}
-\Delta_1u
=w^{-{\lbp}}F_1(|x|,u,|\nabla u|)\mathcal .
\end{equation*}
\end{proposition}

\begin{proof}
\noindent{\bf Step 1.} Fix $R>0$. We claim that
\begin{equation}\label{eq:64-product-limit}
w_j^{c_j}\mathtt{K}_{p_j}(|x|)
\stackrel{*}{\rightharpoonup}
w^{\lbp}\mathtt{K}(|x|)
\quad\text{in }L^\infty(B_R).
\end{equation}

Indeed, let $\phi\in L^1(B_R)$. The local $L^\infty$ bound for
$\mathtt{K}_{p_j}$ supplied by \eqref{eq:Psi-uniform-43}, together with
the a.e. convergence and uniform boundedness of $w_j^{c_j}$ (by  Proposition \ref{44} \ref{52}), gives
\[
\int_{B_R}
\phi\,(w_j^{c_j}-w^{\lbp})\mathtt{K}_{p_j}(|x|)\,\din x
\to   0
\]
by dominated convergence. Also $\phi w^{\lbp}\in L^1(B_R)$, so
\[
\int_{B_R}
\phi w^{\lbp}
\bigl(\mathtt{K}_{p_j}(|x|)-\mathtt{K}(|x|)\bigr)\,\din x
\to   0,
\]
by \eqref{94}. This proves \eqref{eq:64-product-limit}. 

Then
\begin{equation}\label{eq:64-Fp-weak-star}
w_j^{c_j}{\mathtt{K}_{p_j}(|x|)}=F_{p_j}(|x|,u_{p_j},|\nabla u_{p_j}|)
\stackrel{*}{\rightharpoonup}
w^{\lbp}{\mathtt{K}(|x|)}
\quad\text{in }L^\infty(B_R).
\end{equation}

By \eqref{33}, $F_{p_j}\to F_1$ uniformly on compact sets. Moreover,
$u_{p_j}\to u$ uniformly on $B_R$ by \eqref{35}, and $F_1$ is uniformly
continuous on compact sets. Consequently,
\begin{equation}\label{eq:64-uniform-replacement}
\left\|
F_{p_j}(|x|,u_{p_j},|\nabla u_{p_j}|)
-F_1(|x|,u,|\nabla u_{p_j}|)
\right\|_{L^\infty(B_R)}
\to   0.
\end{equation}
It follows from \eqref{eq:64-Fp-weak-star} that
\begin{equation}\label{eq:64-F1-weak-star}
F_1(|x|,u,|\nabla u_{p_j}|)
\stackrel{*}{\rightharpoonup}
w^{\lbp}{\mathtt{K}(|x|)}
\quad\text{in }L^\infty(B_R).
\end{equation}

\medskip
\noindent{\bf Step 2:} {\it the convex case.}
Assume that $\rho\mapsto F_1(r,t,\rho)$ is convex on $[0,\xi_2]$ for
every $(r,t)\in[0,\infty)\times[0,\xi_1]$. Let
\[
\zeta \in C_c^\infty(B_R),
\qquad
{\zeta\geq0}.
\]

By Proposition \ref{44}\ref{51},
\[
{|\nabla u_{p_j}|\rightharpoonup|\nabla u|}
\quad\text{weakly in }L^2(B_R).
\]

By Mazur's lemma \cite{zbMATH05633610}, there are integers $N_m\geq m$ and numbers
$\lambda_{m,j}\geq0$, $m\leq j\leq N_m$, with
\[
\sum_{j=m}^{N_m}\lambda_{m,j}=1,
\]
such that the convex combinations
\[
g _m
:=\sum_{j=m}^{N_m}\lambda_{m,j}|\nabla u_{p_j}|
\]
satisfy
\begin{equation}\label{eq:64-mazur}
g_m \to    |\nabla u|
\quad\text{strongly in }L^2(B_R).
\end{equation}
After passing to a subsequence, we may also assume
\[
g_m\to    |\nabla u|
\quad\text{a.e. in }B_R.
\]

Since every $|\nabla u_{p_j}|$ takes values in $[0,\xi_2]$, the same is true of
$g_m$. Pointwise convexity therefore yields
\begin{equation}\label{eq:64-pointwise-convexity}
F_1(|x|,u(x), g_m(x))
\leq
\sum_{j=m}^{N_m}\lambda_{m,j}
F_1(|x|,u(x),|\nabla u_{p_j}| (x))
\end{equation}
for a.e. $x\in B_R$.

Dominated convergence gives
\begin{equation}\label{eq:64-left-limit}
\int_{B_R}
\zeta F_1(|x|,u,g_m)\,\din x
\to   
\int_{B_R}
\zeta F_1(|x|,u,|\nabla u|)\,\din x.
\end{equation}

From \eqref{eq:64-pointwise-convexity},
\begin{align*}
\int_{B_R}\zeta F_1(|x|,u, g_m)\,\din x
\leq & 
\sum_{j=m}^{N_m}\lambda_{m,j}
\int_{B_R}\zeta F_1(|x|,u,|\nabla u_{p_j}|)\,\din x\\
= & \sum_{j=m}^{N_m}\lambda_{m,j}
\int_{B_R}\zeta \left(F_1(|x|,u,|\nabla u_{p_j}|)-w^{\lbp}{\mathtt{K}(|x|)}\right)\,\din x+\int_{B_R}\zeta w^{\lbp}{\mathtt{K}(|x|)}\,\din x.
\end{align*}

By \eqref{eq:64-F1-weak-star} and \eqref{eq:64-left-limit}, we obtain
\[
\int_{B_R}
\zeta F_1(|x|,u,|\nabla u|)\,\din x
\leq
\int_{B_R}
\zeta w^{\lbp}{\mathtt{K}(|x|)}\,\din x.
\]

This implies
\[
F_1(|x|,u,|\nabla u|)
\leq w^{\lbp}{\mathtt{K}(|x|)}
\qquad\text{for a.e. }x\in B_R.
\]
As $R>0$ was arbitrary, \eqref{eq:convex-bound-43} follows on
$\mathbb R^n$.

\medskip
\noindent{\bf Step 3:} {\it the concave case and the affine case.}
If $\rho\mapsto F_1(r,t,\rho)$ is concave, then
$\rho\mapsto-F_1(r,t,\rho)$ is convex. Applying Step 2 to $-F_1$ and
using \eqref{eq:64-F1-weak-star} gives
\[
-F_1(|x|,u,|\nabla u|)
\leq-w^{\lbp}{\mathtt{K}(|x|)}
\qquad\text{a.e. in }\mathbb R^n,
\]
which is exactly \eqref{eq:concave-bound-43}.

If $\rho\mapsto F_1(r,t,\rho)$ is affine, then it is both convex and
concave. Thus \eqref{eq:convex-bound-43} and
\eqref{eq:concave-bound-43} hold simultaneously, and hence
\[
w^{\lbp}{\mathtt{K}(|x|)}
=F_1(|x|,u,|\nabla u|)
\qquad\text{a.e. in }\mathbb R^n.
\]
This proves \eqref{eq:affine-identification-43}.

\medskip

\noindent{\bf Step 4:} {\it identification of the $1$-Laplacian.} By Proposition \ref{44},
\[
-\operatorname{div}{\bf z}
=\mathtt{K}(|x|),
\qquad
|{\bf z}|\leq1\quad\text{a.e.},
\qquad
({\bf z},Du)=|Du|.
\]
Since $w>0$ a.e., \eqref{eq:affine-identification-43} is equivalent to
\[
{\mathtt{K}(|x|)}
=w^{-{\lbp}}F_1(|x|,u,|\nabla u|)
\qquad\text{a.e. in }\mathbb R^n.
\]
Therefore
\[
-\operatorname{div}{\bf z}
=w^{-{\lbp}}F_1(|x|,u,|\nabla u|).
\]
Together with the preceding bounds and calibration identity, Definition
\ref{113} gives
\[
-\Delta_1u
=w^{-{\lbp}}F_1(|x|,u,|\nabla u|) ,
\]
which is \eqref{eq:affine-limit-problem-43}. \end{proof}

Using $|v_{p_j}'|=\mathtt G_{p_j}^{1/(p_j-1)}$ and $\lim_{t\to0}\log(1+t)/t=1$, we obtain the pointwise limit of $|v_{p_j}'|$ from the assumed limit of $(\mathtt G_{p_j}-1)/(p_j-1)$. Dominated convergence and the previously established weak-* limits then identify $|v'|$ and the nonlinear source almost everywhere on a measurable set $E\subset{\mathtt G=1}$ with $|E|>0$, where the assumed limit exists and is finite almost everywhere.
\begin{lemma}\label{65}
Fix $r>0$ and assume that $\mathtt{G}(r)=1$ and 
\begin{equation}\label{66}
a(r)
:=\lim_{j\to\infty}
\frac{\mathtt{G}_{p_j}(r)-1}{p_j-1}
\in\mathbb R
\end{equation}
exists. Then
\begin{equation}\label{106}
|v_{p_j}'(r)|\to    e^{a(r)}.
\end{equation}

Suppose that $E\subset{s>0\mid \mathrm{G}(s)=1}$ is  measurable with $|E|>0$, and assume that \eqref{66} holds for a.e. $s\in E$. Then 
\begin{equation}\label{108}
|v'(s)|=e^{a(s)}
\qquad\text{for a.e. }s\in E,
\end{equation}
and
\begin{equation}\label{eq:full-source-identification-43}
\mathtt{K}(s)
=w(s)^{-{\lbp}}
F_1\!\left(s,v(s),|v'(s)|\right)
\qquad\text{for a.e. }s\in E.
\end{equation}
Equivalently, for a.e. $x\in\mathbb R^n$ such that $|x|\in E$,
\[
\mathtt{K}(|x|)
=w(x)^{-{\lbp}}
F_1\!\left(|x|,u(x),|\nabla u(x)|\right).
\]
\end{lemma}

\begin{proof}
Set
\[
\delta_j:=\mathtt{G}_{p_j}(r)-1\leq \xi _2^{p_j -1} -1.
\]
By \eqref{66}, $\delta_j/(p_j-1)\to a(r)$; in particular
$\delta_j\to0$. 

Since $\mathtt G_{p_j}(r)\to1$, there is $j_0>0$ such that $\mathtt G_{p_j}(r)>0$ for $j>j_0$. Using \eqref{eq:radial-identities-43},
\[
|v_{p_j}'(r)|
=\mathtt{G}_{p_j}(r)^{\frac{1}{p_j-1}}
=\exp\!\left(
\frac{\log(1+\delta_j)}{p_j-1}
\right).
\]

Define $\sigma:(-1,\infty)\to\mathbb R$ by  $\sigma(t):=t^{-1}\log(1+t)$ if $t\neq0$, and $\sigma(t):=1$ if $t=0$. Since $\sigma(\delta_j)\to1$,
\[
\frac{\log(1+\delta_j)}{p_j-1}
=
\frac{\delta_j}{p_j-1}\sigma (\delta_j)
\to    a(r).
\]
Thus \eqref{106} follows.

The preceding pointwise argument gives
\[
|v_{p_j}'(s)|\to    e^{a(s)}
\qquad\text{for a.e. }s\in E.
\]

Fix $R>0$. By \eqref{34}, $|v_{p_j}'|\leq\xi_2$, and dominated convergence gives
\begin{gather*}
|v_{p_j}'|\stackrel{*}{\rightharpoonup}e^a
\qquad\text{in }L^\infty
\bigl([0,R]\cap E;s^{n-1}\,\din s\bigr).
\end{gather*}
By Proposition \ref{44}\ref{51}, $e^a=|v'|$, and therefore
\eqref{108} holds.

Finally, for a.e. $s\in E$, Proposition \ref{44}\ref{52} gives
\[
v_{p_j}(s)\to v(s),
\qquad
v_{p_j}(s)^{-\beta_{p_j}}
=\bigl(v_{p_j}(s)^{p_j-1}\bigr)^{-c_j}
\to w(s)^{-{\lbp}}.
\]
Thus, by \eqref{33} and \eqref{106},
\[
\begin{aligned}
\mathtt{K}_{p_j}(s)
&=F_{p_j}\!\left(s,v_{p_j}(s),|v_{p_j}'(s)|\right)
   v_{p_j}(s)^{-\beta_{p_j}}\\
&\to   
w(s)^{-{\lbp}}
F_1\!\left(s,v(s),e^{a(s)}\right)\\
&=w(s)^{-{\lbp}}
F_1\!\left(s,v(s),|v'(s)|\right).
\end{aligned}
\]

Lemma \ref{30}\ref{47} gives $\mathtt{K}_{p_j}\leq C_R$ on $[0,R]$, so dominated convergence gives
$$
\mathtt{K}_{p_j}\stackrel{*}{\rightharpoonup}w^{-{\lbp}}
F_1\!\left(s,v(s),|v'(s)|\right)
\quad\text{in }L^\infty([0,R]\cap E;s^{n-1}\din s).
$$
Hence Proposition \ref{44}\ref{56} implies 
\[
{\mathtt{K}(s)=w(s)^{-{\lbp}}F_1\!\left(s,v(s),|v'(s)|\right)}
\quad\text{for a.e. }s\in E,
\]
which is \eqref{eq:full-source-identification-43}.
\end{proof}

The monotonicity of $r\mapsto F_{p_j}(r,0,0)$ yields a lower bound for $\mathtt G_{p_j}(r)$. At points where $F_1(r,0,0)=n/r$, a first-order expansion gives a lower bound for $\liminf_{j\to\infty}(\mathtt G_{p_j}(r)-1)/(p_j-1)$, provided the finite limit in \eqref{75} exists. The uniform derivative bound gives an upper bound for the corresponding limit superior at every $r>0$.
\begin{lemma}\label{67}
Under the notation and conclusions of Proposition 
\ref{44}, the following statements hold.
\begin{enumerate}[label=$(\roman*)$]
\item \label{69} Assume that, for all sufficiently large $j$,
\begin{equation}\label{eq:optional-monotonicity-43}
r\mapsto  
F_{p_j}(r,0,0) 
\qquad\text{is nonincreasing on }[0,\infty).
\end{equation}
Then, for every $r>0$ and all sufficiently large $j$,
\begin{equation}\label{72}
\mathtt{G}_{p_j}(r)
\geq
\frac{r}{n}\,\xi_1^{-\beta_{p_j}}F_{p_j}(r,0,0).
\end{equation}
Consequently,
\begin{equation}\label{73}
\frac{r}{n}F_1(r,0,0)
\leq \mathtt{G}(r)\leq1,
\qquad r>0.
\end{equation}

\item \label{70} Assume \eqref{eq:optional-monotonicity-43}. If, at a
point $r>0$ satisfying $F_1(r,0,0)=n/r$, the finite limit
\begin{equation}\label{75}
d(r):=
\lim_{j\to\infty}
\frac{F_{p_j}(r,0,0)-F_1(r,0,0)}{p_j-1}
\in\mathbb R
\end{equation}
exists, then
\begin{equation}\label{76}
\frac{r}{n}\xi_1^{-\beta_{p_j}}F_{p_j}(r,0,0)
=
1+(p_j-1)
\left(
\frac{r}{n}d(r)-{\lbp}\log\xi_1
\right)
+o(p_j-1),
\end{equation}
and therefore
\begin{equation}\label{77}
\frac{r}{n}d(r)-{\lbp}\log\xi_1
\leq
\liminf_{j\to\infty}
\frac{\mathtt{G}_{p_j}(r)-1}{p_j-1}.
\end{equation}

\item \label{71} For every fixed $r>0$,
\begin{equation*}\label{eq:critical-upper-rate-67}
\limsup_{j\to\infty}
\frac{\mathtt{G}_{p_j}(r)-1}{p_j-1}
\leq \log\xi_2.
\end{equation*}
\end{enumerate}
\end{lemma}

\begin{proof}

\noindent\ref{69}
Fix $0<s\leq r$ and a sufficiently large $j$.  By \eqref{eq:optional-monotonicity-43}, condition \ref{2}, and $v_{p_j}\leq\xi_1$, we have
\[
\mathtt{K}_{p_j}(s)={v_{p_j}(s)^{-\beta_{p_j}}}F_{p_j}\!\left(s,v_{p_j}(s),|v_{p_j}'(s)|\right)
\geq
\xi_1^{-\beta_{p_j}}F_{p_j}(s,0,0)
\geq
\xi_1^{-\beta_{p_j}}F_{p_j}(r,0,0).
\]

Using the definition of $\mathtt{G}_{p_j}$,
\begin{align*}
\mathtt{G}_{p_j}(r)
&=r^{1-n}\int_0^r s^{n-1}\mathtt{K}_{p_j}(s)\,\din s\\
&\geq
r^{1-n}\xi_1^{-\beta_{p_j}}F_{p_j}(r,0,0)
\int_0^r s^{n-1}\,\din s\\
&=
\frac{r}{n}\xi_1^{-\beta_{p_j}}F_{p_j}(r,0,0),
\end{align*}
which proves \eqref{72}.

Now let $j\to\infty$. By \eqref{eq:Ap-limit-43},
$\mathtt{G}_{p_j}(r)\to\mathtt{G}(r)$. By \eqref{33},
$F_{p_j}(r,0,0)\to F_1(r,0,0)$, while \eqref{32} implies
$\beta_{p_j}\to0$.
Passing to the limit in \eqref{72} yields
\[
\frac{r}{n}F_1(r,0,0)\leq\mathtt{G}(r).
\]
The bound $\mathtt{G}\leq1$ is part of Proposition \ref{44}. This proves
\eqref{73}.

\medskip

\noindent\ref{70}
Put $h_j:=p_j-1$. From \eqref{75},
\[
F_{p_j}(r,0,0)
=F_1(r,0,0)+h_jd(r)+o(h_j).
\]
Since $rF_1(r,0,0)=n$,
\begin{equation}\label{eq:F-expansion-67}
\frac{r}{n}F_{p_j}(r,0,0)
=1+h_j\frac{r}{n}d(r)+o(h_j).
\end{equation}
On the other hand, \eqref{32} gives
\[
\beta_{p_j}=h_j{\lbp}+o(h_j).
\]
Consequently,
\begin{align*}
\xi_1^{-\beta_{p_j}}
&=\exp\!\left(-\beta_{p_j}\log\xi_1\right)\\
&=1-h_j{\lbp}\log\xi_1+o(h_j).
\end{align*}
Multiplying this expansion by \eqref{eq:F-expansion-67} proves \eqref{76}.

Finally, \eqref{72} gives
\[
\frac{\mathtt{G}_{p_j}(r)-1}{h_j}
\geq
\frac{1}{h_j}
\left(
\frac{r}{n}\xi_1^{-\beta_{p_j}}F_{p_j}(r,0,0)-1
\right).
\]
Taking the limit inferior and using \eqref{76} proves \eqref{77}.

\medskip

\noindent\ref{71}
Equations \eqref{eq:radial-identities-43} and \eqref{34} give
\[
\mathtt{G}_{p_j}(r)
=|v_{p_j}'(r)|^{p_j-1}
\leq \xi_2^{p_j-1}.
\]
Therefore,
\[
\frac{\mathtt G_{p_j}(r)-1}{p_j-1}
\leq
\frac{\xi_2^{p_j-1}-1}{p_j-1}
\to   \log\xi_2,
\]
which proves \ref{71}.
\end{proof}

\subsection{Proofs of the Main Results}

\begin{proof}[Proof of Theorem \ref{80}]
Proposition \ref{44} gives, after extraction,
\eqref{eq:80-compactness}--\eqref{eq:80-complementarity}.
Lemma \ref{82}\ref{62} gives
\eqref{95}. Lemma
\ref{82}\ref{63} gives
\eqref{eq:80-origin}--\eqref{eq:80-origin-phase}.
\end{proof}

\begin{proof}[Proof of Theorem \ref{81}]
By \eqref{eq:80-divergence}, \eqref{eq:80-calibration}, and Definition
\ref{113},
\[
 -\Delta_1u
 =\mathtt K(|x|),
\]
which proves \eqref{eq:81-global-1laplace}.

\medskip
\noindent\ref{thm81-affine} This follows from Proposition \ref{64}.

\medskip
\noindent\ref{thm81-critical} This follows from Lemma \ref{82} and Lemma \ref{65}. \end{proof}

%%
%%
%%
%%
%%%
%%
%%

%%
%%
%%
%%
%%%
%%
%%

%%
%%
%%
%%
%%%
%%
%%

%%
%%
%%
%%
%%%
%%
%%

%%
%%
%%
%%
%%%
%%
%%

%%
%%
%%
%%
%%%
%%
%%

%%
%%
%%
%%
%%%
%%
%%

%%
%%
%%
%%
%%%
%%
%%

\section{The one-dimensional case $n=1$}\label{96}

In this section, we work with
\[
n=1,\qquad 1<p<2.
\]

%Existence theory for one-dimensional $p$-Laplacian equations with $p>1$ has been developed by several methods, including Leray--Schauder degree, upper and lower solutions, and fixed-point arguments; see \cite{DelPinoElguetaManasevich1989,JiangGao2002,Liu2010}. The distinction between the singular interval $1<p<2$ and the degenerate interval $p>2$ is standard; see \cite[Chapters 1--4]{Lindqvist2019}.

Let
\[
\ell\geq0,\qquad
\xi_1\geq1,\qquad
\xi_2>0,
\]
and, for every fixed $p>1$, let
\[
0\leq\beta_p\leq p-1.
\]

For $h\in C([0,\infty))$, $h\geq 0$, define
\begin{equation}\label{eq:n-one-Theta}
(\Theta_{p}h)(t)
:=
\int_0^t
\left[
\int_0^s h(r)\,\din r
\right]^{{\qpu}   }\din s,
\qquad t\geq0.
\end{equation}

\begin{lemma}\label{lem:n-one-estimates}

Let $p>1$, let $h\in C([0,\infty))$ with $h\geq0$, and suppose that
\[
M(h):=\int_0^\infty h(t)\,\din t<\infty.
\]
Then
\begin{equation}\label{eq:n-one-Theta-estimates}
0\leq(\Theta_{p}h)'(t)
\leq M(h)^{\qpu}   ,
\qquad
0\leq(\Theta_{p}h)(t)
\leq M(h)^{\qpu}    \ftes (t) ,
\qquad t\geq0,
\end{equation}
where $\ftes (t) :=\max \{\ell , t\}$.
\end{lemma}

\begin{proof}
Since $h\geq0$,
\[
0\leq
\int_0^t h(r)\,\din r
\leq M(h).
\]
Therefore,
\[
0\leq
(\Theta_{p}h)'(t)
=
\left[
\int_0^t h(r)\,\din r
\right]^{\qpu}   
\leq M(h)^{\qpu}   .
\]
Integrating from $0$ to $t$ gives
\[
0\leq
(\Theta_{p}h)(t)
\leq tM(h)^{\qpu}   
\leq\max\{\ell,t\}M(h)^{\qpu}   
=\ftes (t) M(h)^{\qpu}   .
\]
\end{proof}

The following assumptions are motivated by conditions (F1), (F2a), and (F3a$\infty$) in \cite{Qi2010}. Throughout, we assume that:
\begin{enumerate}[label=$(\fnu_{\arabic*})$]
\item \label{83}
There exists $\theta\in(0,1]$ such that $\fnu _p\in
C_{\loc}^{0,\theta}
([0,\infty)^3)$, and $\fnu _p(t,u,s)\geq 0$.

\item \label{84}
For every fixed $t\geq0$, the function $u\mapsto   \fnu _p(t,u,s)$ is nondecreasing on $[1,\xi_1(1+\ftes (t) )]$ for every $s\in[0,\xi_2]$, and  $s\mapsto   \fnu _p(t,u,s)$
is nondecreasing on $[0,\xi_2]$ for every $u\in[1,\xi_1(1+\ftes (t) )]$.

\item \label{85} With $m_\xi:=\min\{\xi_1,\xi_2\}$, one has
\begin{equation}\label{eq:n-one-structural-less-two}
\int_0^\infty
\fnu _p\left(
t,\xi_1(1+\ftes (t) ),\xi_2
\right)\,\din t
\leq
m_\xi^{p-1}.
\end{equation}
\end{enumerate}

\begin{proposition}\label{prop:n-one-less-two}
Assume conditions \ref{83}--\ref{85}. Then
\eqref{117} has at least one even positive entire
solution
\[
u_p(x)=y_p(|x|).
\]
It satisfies
\begin{equation}\label{eq:n-one-bounds-less-two}
1\leq y_p(t)\leq\xi_1(1+\ftes (t) ),
\qquad
0\leq y_p'(t)\leq\xi_2,
\qquad t\geq0,
\end{equation}
and
\[
y_p(0)=1,\qquad y_p'(0)=0.
\]

Moreover, if
\begin{equation}\label{97}
\fnu _p(r_0,1,0) \neq 0 \quad \text{for some } r_0\in[0,\infty),
\end{equation}
then
\begin{equation}\label{eq:n-one-asymptotic-less-two}
\lim_{t\to\infty}\frac{y_p(t)}{\ftes (t) }
=C_{p},
\end{equation}
where
\begin{equation}\label{eq:n-one-C-less-two}
C_p   
=
\left[
\int_0^\infty
\fnu _p(t,y_p(t),y_p'(t))
y_p(t)^{-\beta_p}\,\din t
\right]^{\frac1{p-1}}
\end{equation}
and $0<C_p   \leq m_\xi$.
\end{proposition}

\begin{proof}
 \noindent {\bf Step 1.} Put
\[
U(t):=\xi_1(1+\ftes (t) ),
\qquad
G_p(t):=\fnu _p(t,U(t),\xi_2).
\]

Condition \ref{85} gives
\begin{equation}\label{eq:n-one-H-less-two}
\left(\int_0^\infty G_p(t)\,\din t \right)  ^{\qpu}   \leq m_\xi.
\end{equation}

Let $X=C_{\loc}^1([0,\infty))$ with its usual Fr\'echet topology, and define
\[
Q=
\left\{
y\in X \mid 
1\leq y(t)\leq U(t),\
0\leq y'(t)\leq\xi_2,\
t\geq0
\right\}.
\]
The constant function $y(t)\equiv1$ belongs to $Q$; hence $Q$ is
nonempty. It is also closed and convex.

For $y\in Q$, set
\[
g_y(t)
:=
\fnu _p(t,y(t),y'(t))y(t)^{-\beta_p}
\]
and
\begin{equation}\label{eq:n-one-T-less-two}
(Ty)(t):=1+(\Theta_{p}g_y)(t).
\end{equation}

Since $y\geq1$, one has $y^{-\beta_p}\leq1$. By
condition \ref{84},
\begin{equation}\label{eq:n-one-gy-less-two}
g_y(t)
\leq
\fnu _p(t,U(t),\xi_2)
=G_p(t).
\end{equation}
Lemma \ref{lem:n-one-estimates} and
\eqref{eq:n-one-H-less-two} imply
\begin{align*}
1\leq(Ty)(t)
&\leq
1+\left(\int_0^\infty G_p(t)\,\din t \right)  ^{\qpu}    \ftes (t) \\
&\leq
1+m_\xi \ftes (t) \\
&\leq
\xi_1(1+\ftes (t) ),
\end{align*}
where the final inequality uses $\xi_1\geq1$ and
$m_\xi\leq\xi_1$. Furthermore,
\[
0\leq(Ty)'(t)
\leq
\left(\int_0^\infty G_p(t)\,\din t \right)  ^{\qpu}   
\leq
m_\xi
\leq\xi_2.
\]
Thus $T(Q)\subset Q$.

\medskip 

\noindent {\bf Step 2.} We next prove continuity. Suppose that
$y_j,y\in Q$ and $y_j\to y_0$ in $X$, and fix $R>0$. Then
\[
\epsilon  _{j,R}
:=
\sup_{0\leq t\leq R}
|g_{y_j}(t)-g_{y_0}(t)|
\to   0.
\]

Define
\[
A_j(t):=\int_0^t g_{y_j}(r)\,\din r,
\qquad j=0,1,2, \ldots.
\]

Then
\begin{equation}\label{eq:n-one-A-continuity-less-two}
\sup_{0\leq t\leq R}|A_j(t)-A_0 (t) |
\leq
R\epsilon  _{j,R}.
\end{equation}

From \eqref{eq:n-one-gy-less-two}, 
\[
0\leq A_j(t),A_0 (t) \leq c_R R, \qquad c_R:= \max_{0\leq t\leq R}G_p(t).
\]

Since ${\qpu}   >1$, the mean value theorem yields
\[
|r^{\qpu}   -s^{\qpu}   |
\leq
{\qpu}   (c_RR)^{{\qpu}   -1}|r-s|,
\qquad
0\leq r,s\leq c_RR.
\]

Consequently,
\begin{align*}
\sup_{0\leq t\leq R}
|(Ty_j)'(t)-(Ty)'(t)|
&\leq
{\qpu}   (c_RR)^{{\qpu}   -1}
R\epsilon  _{j,R}
\to   0,\\
\sup_{0\leq t\leq R}
|(Ty_j)(t)-(Ty)(t)|
&\leq
{\qpu}    R^2(c_RR)^{{\qpu}   -1}
\epsilon  _{j,R}
\to   0.
\end{align*}

Hence $T:Q\to Q$ is continuous in $X$.

\medskip

\noindent {\bf Step 3.} To prove relative compactness, observe that, for $y\in Q$ and
$s,t\in[0,R]$,
\[
|A_0(t)-A_0(s)|
=
\left|
\int_s^t g_y(r)\,\din r
\right|
\leq c_R|t-s|.
\]

Therefore,
\begin{equation}\label{eq:n-one-derivative-Lipschitz-less-two}
|(Ty)'(t)-(Ty)'(s)|
\leq
{\qpu}   (c_RR)^{{\qpu}   -1}c_R|t-s|.
\end{equation}

The families $\{Ty \mid y\in Q\}$ and $\{(Ty)'\mid y\in Q\}$ are uniformly bounded and equicontinuous on
$[0,R]$. The Arzel\`a--Ascoli theorem and a diagonal argument imply
that $T(Q)$ is relatively compact in $X$.

The Schauder--Tychonoff theorem
\cite[Theorem 10.1]{zbMATH07063009} now gives a fixed point
$y_p\in Q$. Thus
\[
y_p(t)
=
1+\int_0^t
\left[
\int_0^s
\fnu _p(r,y_p(r),y_p'(r))
y_p(r)^{-\beta_p}\,\din r
\right]^{\qpu}   \din s.
\]

It follows that
\begin{equation}\label{98}
|y_p'(t)|^{p-2}y_p'(t)
=
\int_0^t
\fnu _p(r,y_p(r),y_p'(r))
y_p(r)^{-\beta_p}\,\din r
\end{equation}
and hence
\begin{equation}\label{eq:n-one-halfline-equation-less-two}
\left(
|y_p'|^{p-2}y_p'
\right)'
=
\fnu _p(t,y_p,y_p')y_p^{-\beta_p},
\qquad t>0.
\end{equation}

Since $y_p'(0)=0$, the even extension $u_p(x)=y_p(|x|)$ belongs to $C^1(\mathbb R)$. By \eqref{98} and the continuity of its integrand, $|u_p'|^{p-2}u_p'\in C^1(\mathbb R)$, and $u_p$ satisfies \eqref{117} on $\mathbb R$.

Finally, set
\[
I_p:=
\int_0^\infty
\fnu _p(t,y_p(t),y_p'(t))
y_p(t)^{-\beta_p}\,\din t.
\]

By \eqref{eq:n-one-gy-less-two},
\[
I_p\leq {\int_0^\infty G_p(t)\,\din t }  <\infty,
\]
and
\[
y_p'(t)
=
\left[
\int_0^t
\fnu _p(r,y_p(r),y_p'(r))
y_p(r)^{-\beta_p}\,\din r
\right]^{\qpu}   
\to    I_p^{\qpu}  .
\]

By \eqref{97} and condition \ref{83}, $\fnu _p(r_0,1,0)>0$. Hence there
exist $\epsilon>0$ and $c_0>0$ such that
\[
\fnu _p(t,1,0)\geq c_0,
\qquad t\in [r_0,r_0+\epsilon].
\]
Condition \ref{84} and \eqref{eq:n-one-bounds-less-two} give
\[
I_p
\geq
c_0\int^{r_0+\epsilon} _{r_0} U(t)^{-\beta_p}\,\din t
>0.
\]
Thus $y_p'(t)\to I_p^{\qpu}>0$, so $y_p(t)\to\infty$.

For $t>\ell$, $\ftes (t) =t$ and $a'(t)=1$. Therefore, l'Hospital's rule gives
\[
\lim_{t\to\infty}\frac{y_p(t)}{\ftes (t) }
=
\lim_{t\to\infty}y_p'(t)
=I_p^{\qpu}   .
\]
This is \eqref{eq:n-one-asymptotic-less-two}--%
\eqref{eq:n-one-C-less-two}, and
\[
0<C_p   =I_p^{\qpu}   
\leq \left(\int_0^\infty G_p(t)\,\din t \right)  ^{\qpu}   
\leq m_\xi.
\]
\end{proof}

\medskip

\noindent {\bf Example for $\fnu _p$.} Define
\begin{equation*}\label{b4}
J_\ell:=(1+\ell)^2+2(\ell+2)e^{-\ell},\qquad
\kappa:=\frac{\min\{1,m_\xi\}}
{(1+\xi_2^2)(1+\xi_1^2J_\ell)},
\end{equation*}
and
\begin{equation*}\label{b5}
\mathtt F_p(t,u,s)=\mathtt F_1(t,u,s)
:=\kappa e^{-t}(1+u^2)(1+s^2),\qquad t,u,s\geq0.
\end{equation*}

Then \ref{83}--\ref{85} hold for every $1<p<2$. 

%
%%
%%
%
%%
%%
%%%
%%
%%
%%%
%%%
%%

%
%%
%%
%
%%
%%
%%%
%%
%%
%%%
%%%
%%

\section{The limit $p\downarrow1$ when $n=1$}
\label{100}

We assume the following conditions.
\begin{enumerate}[label=$(L_{\arabic*})$]
\item\label{87}
There exists $F_1\in C^{0,\theta}_{\loc}([0,\infty)^3)$ such that
\begin{equation}\label{eq:n-one-F-local-limit}
F_p\to    F_1
\qquad\text{locally uniformly in }[0,\infty)^3,
\end{equation}
and
\begin{equation}\label{eq:n-one-beta-limit}
\frac{\beta_p}{p-1}\to   {\lbp}\in[0,1].
\end{equation}

\item\label{88} $F_1(r_0,t_0,\rho_0)>0 $ for some $(r_0,t_0,\rho_0) \in{(0,\infty)}\times[1,\xi_1(1+\ell)]\times[0,\xi_2]$.

\end{enumerate}

For every $p\in (1,2)$, Proposition \ref{prop:n-one-less-two} gives
\[
u_{p}(x)=y_{p}(|x|),
\qquad
y_{p}(0)=1,
\qquad
y_{p}'(0)=0,
\]
with
\begin{equation}\label{114}
1\leq y_{p}(t)\leq U(t):=\xi_1(1+a(t)),
\qquad
0\leq y_{p}'(t)\leq\xi_2.
\end{equation}

Set \label{116}
\begin{gather}
k_{p}(t) :=F_{p}(t,y_{p}(t),y_{p}'(t))y_{p}(t)^{-\beta_p},
\label{eq:n-one-kj}\\
G_{p}(t):=\int_0^t k_{p}(s)\,\din s
=y_{p}'(t)^{p-1}.
\label{eq:n-one-Gj}
\end{gather}

We have,
\begin{equation}\label{eq:n-one-kj-envelope}
0\leq k_{p}(t)\leq F_{p}(t,U(t),\xi_2),
\qquad t\geq0.
\end{equation}

\subsection{Uniform estimates and subsequential convergence}

We first extract convergent subsequences of the functions $y_{p_j}$, their derivatives $y_{p_j}'$, and the integrated sources $G_{p_j}$. We then use the monotonicity of $G$ to describe the set $\{r>0 \mid G(r)=1\}$. These results provide the limit equation. An additional assumption on the convergence of $(G_{p_j}-1)/(p_j-1)$ is used to identify the nonlinear source on this set.

\begin{proposition}\label{101}
Assume \ref{83}--\ref{85}, and \ref{87}. After passing to a subsequence, there exist $y\in W^{1,\infty}_{\loc}([0,\infty))$, $G\in W^{1,\infty}_{\loc}([0,\infty))$, and $K\in L^\infty_{\loc}([0,\infty))$ such that
\begin{enumerate}[label=$(\roman*)$]
\item \label{102} \begin{gather}
y_{p_j}\to    y
\quad\text{in }C_{\loc}([0,\infty)),
\qquad
y_{p_j}'\stackrel{*}{\rightharpoonup}y'
\quad\text{in }L^\infty_{\loc}(0,\infty),
\label{eq:n-one-general-y-limit}\\
G_{p_j}\to    G
\quad\text{in }C_{\loc}([0,\infty)),
\qquad
k_{p_j}\stackrel{*}{\rightharpoonup}K
\quad\text{in }L^\infty_{\loc}(0,\infty),
\label{eq:n-one-general-flux-limit}\\
G'=K\geq0,
\qquad
G(0)=0,
\qquad
0\leq G\leq1,
\label{eq:n-one-general-G-properties}\\
(1-G)y'=0
\qquad\text{a.e. on }(0,\infty).
\label{105}
\end{gather}

\item \label{103} Define $\mathscr{C}_1:=\{t>0 \mid G(t)=1\}$. If $\mathscr{C}_1\neq\emptyset$, then
\begin{equation}\label{eq:n-one-general-phases}
\left\{\begin{aligned}
&G(t)<1 &&0\leq t< \inf\mathscr{C}_1,\\
&G(t)=1 &&t\geq \inf\mathscr{C}_1.
\end{aligned}\right.
\end{equation}

Moreover,  for every $r>0$ such that $G(r)<1$,
\begin{equation}\label{eq:n-one-general-subcritical-identification}
y(t)=1,
\qquad
K(t)=G'(t)=F_1(t,1,0)
\quad\text{for a.e. }t\in(0,r).
\end{equation}

If $\mathscr{C}_1\neq\emptyset$, then
\begin{equation}\label{eq:n-one-general-critical-source}
K(t)=0
\quad\text{for a.e. }t\in( \inf\mathscr{C}_1 ,\infty).
\end{equation}

Let
\[
u(x):=y(|x|),
\qquad
z(x):=\operatorname{sign}(x)G(|x|),
\qquad
z(0):=0.
\]

Then
\begin{equation}\label{eq:n-one-general-one-laplacian}
\Delta_1u=K(|x|)
\qquad\text{in }\mathbb R
\end{equation}
in the sense of Definition \ref{113}.

\item \label{104} Assume that $\mathscr{C}_1\neq\emptyset$, and let $E\subset(\inf\mathscr{C}_1,\infty)$ be measurable. Suppose that
\begin{equation*}\label{eq:n-one-critical-rate}
a(t):=\lim_{j\to\infty}
\frac{G_{p_j}(t)-1}{p_j-1}\in\mathbb R
\qquad\text{for a.e. }t\in E.
\end{equation*}

Then
\begin{equation}\label{107}
y'(t)=e^{a(t)}
\qquad\text{for a.e. }t\in E,
\end{equation}
and
\begin{equation}\label{eq:n-one-critical-F-zero}
F_1(t,y(t),y'(t))=0
\qquad\text{for a.e. }t\in E.
\end{equation}
\end{enumerate}
\end{proposition}

\begin{proof}
\noindent \ref{102}  Fix $R>0$. The bounds supplied by Proposition
\ref{prop:n-one-less-two} give
\[
1\leq y_{p_j}(t)\leq U(R),
\qquad
0\leq y_{p_j}'(t)\leq\xi_2,
\qquad 0\leq t\leq R.
\]
Arzel\`a--Ascoli and a diagonal argument yield
\eqref{eq:n-one-general-y-limit}. By
\eqref{eq:n-one-F-local-limit},
\eqref{eq:n-one-kj}, and $y_{p_j}^{-\beta_j}\leq1$,
\[
0\leq k_{p_j}(t)\leq C_R,
\qquad 0\leq t\leq R,
\]
where $C_R$ is independent of $j$. Since
\[
G_{p_j}'=k_{p_j},
\qquad
G_{p_j}(0)=0,
\qquad
0\leq G_{p_j}\leq\xi_2^{p_j-1},
\]
Arzel\`a--Ascoli, Banach--Alaoglu, and a diagonal argument give
\eqref{eq:n-one-general-flux-limit} and
\eqref{eq:n-one-general-G-properties}.

Equation \eqref{eq:n-one-Gj} gives
\[
G_{p_j}y_{p_j}'=(y_{p_j}')^{p_j}.
\]

Set
\[
\delta_j:=\sup_{0\leq s\leq\xi_2}|s^{p_j}-s|.
\]
Then
\[
\delta_j\to   0,
\qquad
|G_{p_j}y_{p_j}'-y_{p_j}'|\leq\delta_j.
\]
The locally uniform convergence $G_{p_j}\to G$ and the weak-* convergence
$y_{p_j}'\stackrel{*}{\rightharpoonup}y'$ imply
\[
Gy'=y'
\qquad\text{a.e. on }(0,\infty),
\]
which proves \eqref{105}.

\medskip

\noindent \ref{103} Since $G$ is continuous and nondecreasing,
$\mathscr{C}_1\neq\emptyset$ implies
\eqref{eq:n-one-general-phases}.

Fix $r>0$ such that $G(r)<1$. Since $G$ is nondecreasing,
$G(t)\leq G(r)<1$ for $t\in[0,r]$. Hence \eqref{105} gives
$y'=0$ a.e. on $(0,r)$ and thus $y=1$ on $[0,r]$. For all
sufficiently large $j$,
\[
0\leq G_{p_j}(t)
\leq G(t)+\frac{1-G(r)}{2}
\leq\frac{1+G(r)}{2}<1,
\qquad t\in[0,r].
\]
Therefore,
\[
0\leq y_{p_j}'(t)
=G_{p_j}(t)^{1/(p_j-1)}
\leq\left(\frac{1+G(r)}{2}\right)^{1/(p_j-1)}
\to   0
\]
uniformly on $[0,r]$. Moreover,
\[
0\leq\beta_j\log y_{p_j}(t)
\leq(p_j-1)\log U(r)
\to   0
\]
uniformly on $[0,r]$. Thus $y_{p_j}^{-\beta_j}\to1$ uniformly on
$[0,r]$, and \eqref{eq:n-one-F-local-limit} yields
\[
k_{p_j}\to    F_1(\cdot,1,0)
\qquad\text{uniformly on }[0,r].
\]
Together with \eqref{eq:n-one-general-flux-limit}, this proves
\eqref{eq:n-one-general-subcritical-identification}.

If $\mathscr{C}_1\neq\emptyset$, then $G=1$ on
$(\inf\mathscr{C}_1  ,\infty)$, so $K=G'=0$ there. This
proves \eqref{eq:n-one-general-critical-source}.

Since $G(0)=0$, one has
$z\in W^{1,\infty}_{\loc}(\mathbb R)$ and $z'=K(|x|)$ a.e. in
$\mathbb R$. Moreover, by \eqref{105},
\[
zu'=G(|x|)|u'|=|u'|
\quad\text{a.e. in }\mathbb R.
\]
Hence
\[
|z|\leq1,
\qquad
(z,Du)=|Du|,
\]
see \eqref{93}. This proves \eqref{eq:n-one-general-one-laplacian}.

\medskip

\noindent \ref{104} For a.e. $t\in E$, one has
$G_{p_j}(t)\to1$ and $G_{p_j}(t)>0$ for all sufficiently large $j$. The   calculation used in the proof of \eqref{106} gives
\[
y_{p_j}'(t)
=\exp\left(\frac{\log G_{p_j}(t)}{p_j-1}\right)
\to    e^{a(t)}.
\]

For every $R>0$, the bound $0\leq y_{p_j}'\leq\xi_2$, dominated
convergence on $E\cap(0,R)$, and
\eqref{eq:n-one-general-y-limit} imply
\[
y'(t)=e^{a(t)}
\qquad\text{for a.e. }t\in E.
\]

Furthermore, for a.e. $t\in E\cap(0,R)$,
\[
y_{p_j}(t)\to y(t),
\qquad
y_{p_j}'(t)\to y'(t),
\qquad
0\leq\beta_j\log y_{p_j}(t)
\leq(p_j-1)\log U(R)\to0.
\]
Consequently,
\[
k_{p_j}(t)\to    F_1(t,y(t),y'(t))
\qquad\text{for a.e. }t\in E.
\]
The sequence $\{k_{p_j}\}$ is uniformly bounded on $[0,R]$.
Dominated convergence and \eqref{eq:n-one-general-flux-limit} give
\[
K(t)=F_1(t,y(t),y'(t))
\qquad\text{for a.e. }t\in E.
\]
Since \eqref{eq:n-one-general-critical-source} gives $K=0$ a.e. on
$E$, \eqref{eq:n-one-critical-F-zero} follows.
\end{proof}

\subsection{Behavior of the limit and source identification}
The condition $F_1(r_0,t_0,\rho_0)>0$ in \ref{88}, together with the total integral bound, yields $0\leq G_{p_j}\leq\delta$ on $[0,r_0]$ for some $\delta\in(0,1)$ and all sufficiently large $j$. Since $y_{p_j}'=G_{p_j}^{1/(p_j-1)}$, this estimate gives $u_{p_j}\to1$ in $C^1([-r_0,r_0])$. The convergence of $y_{p_j}$ and $y_{p_j}'$ then allows us to identify the limiting source as $F_1(\cdot,1,0)$.
\begin{theorem}\label{thm:n-one-p-to-one}
Assume \ref{83}--\ref{85},\ref{87}, and \ref{88}. There exist $\delta  \in(0,1)$ and  $j_{0}    \in\mathbb N$, such that
\begin{equation}\label{eq:n-one-exponential-collapse}
\|y_{p_j}'\|_{L^\infty(0,{r_0}    )}
\leq \delta  ^{1/(p_j-1)},
\qquad
\|y_{p_j}-1\|_{L^\infty(0,{r_0}    )}
\leq {r_0}    \delta  ^{1/(p_j-1)}
\end{equation}
for every $j\geq j_{0}    $.

 Consequently,
\begin{gather}
u_{p_j}\to   1
\qquad\text{in }C^1 ([-r_0,r_0] ) \label{eq:n-one-C1-limit},\\
u_{p_j}^{p_j-1}\to   1,
\qquad
u_{p_j}^{-\beta_j}\to   1
\qquad\text{in }C ([-r_0,r_0]     ), \label{eq:n-one-power-limits}\\
k_{p_j}\to    F_1(\cdot,1,0)
\qquad\text{in }C  ([0,r_0]). \label{eq:n-one-source-limit}
\end{gather}

Furthermore,
\begin{gather}
G_{p_j}\to    \int_0^tF_1(s,1,0)\,\din s
\qquad\text{in }C^1([0,r_0]),\label{eq:n-one-G-C1-limit}\\
{0\leq \int_0^tF_1(s,1,0)\,\din s<1
\qquad\text{for every }t\in[0,r_0],}\label{eq:n-one-subcritical-limit}
\end{gather}

For
\begin{equation}\label{eq:n-one-zj}
z_{p_j}(x):=|u_{p_j}'(x)|^{p_j-2}u_{p_j}'(x)
=\operatorname{sign}(x)G_{p_j}(|x|),
\qquad z_{p_j}(0):=0,
\end{equation}
one has
\begin{equation}\label{eq:n-one-z-limit}
z_{p_j}\to    z
\qquad\text{in }C^1([-r_0,r_0]     ),
\qquad
z(x)=\operatorname{sign}(x)G(|x|),
\qquad z(0)=0,
\end{equation}
and
\begin{equation}\label{eq:n-one-limit-flux-equation}
z'(x)=F_1(|x|,1,0),
\qquad
|z(x)|<1,
\qquad x\in[-r_0,r_0]     .
\end{equation}

Thus,  
\begin{equation}\label{eq:n-one-limit-one-laplacian}
\Delta_1u
=F_1(|x|,u,|u'|) 
=F_1(|x|,1,0) 
\qquad\text{in }{ (-r_0,r_0)}
\end{equation}
in the local sense of Definition \ref{113}.
\end{theorem}

\begin{proof}
\noindent\textbf{Step 1.} Local uniform convergence implies that
$F_1$ inherits the monotonicity in \ref{84}. By \ref{84}, \ref{87}, and \ref{88}, there exist $r_1>r_0$ and $j_0\in\mathbb N$ such that,
for every $t\in[r_0,r_1]$ and $j\geq j_0$,
\begin{gather*}
F_1(t,U(t),\xi_2)
\geq F_1(t,t_0,\rho_0)
>\frac12F_1(r_0,t_0,\rho_0)>0.
\end{gather*}

By \eqref{eq:n-one-F-local-limit},
\[
\int_{r_0}^{r_1}F_{p_j}(t,U(t),\xi_2)\,\din t
\to   
c_0:=\int_{r_0}^{r_1}F_1(t,U(t),\xi_2)\,\din t>0.
\]

Condition \ref{85} and
$m_\xi^{p_j-1}\to1$ give
\[
0<c_0    
=\lim_{j\to\infty}\int_{r_0}    ^{{r_1}}F_{p_j}  (t,U(t),\xi_2)\,\din t
\leq\limsup_{j\to\infty}\int_0^\infty F_{p_j}  (t,U(t),\xi_2)\,\din t  
\leq1.
\]

Hence, for all sufficiently large $j$,
\[
\int_0^\infty F_{p_j}  (t,U(t),\xi_2)\,\din t  \leq1+\frac{c_0    }{4},
\qquad
\int_{r_0}    ^{{r_1}}F_{p_j}  (t,U(t),\xi_2)\,\din t\geq\frac{3c_0    }{4}.
\]

Therefore,
\begin{equation}\label{eq:n-one-compact-flux-gap}
\begin{aligned}
\int_0^{r_0}     F_{p_j}  (t,U(t),\xi_2)\,\din t
\leq & \int_0^\infty F_{p_j}  (t,U(t),\xi_2)\,\din t  -\int_{r_0}    ^{{r_1}}F_{p_j}  (t,U(t),\xi_2)\,\din t
\\
\leq & 1-\frac{c_0    }{2}
=:\delta {\in(0,1)}.
\end{aligned}
\end{equation}

For $0\leq t\leq {r_0}    $, \eqref{eq:n-one-Gj},
\eqref{eq:n-one-kj-envelope}, and
\eqref{eq:n-one-compact-flux-gap} yield
\[
0\leq G_{p_j}(t)\leq \delta  .
\]

Thus
\begin{equation}\label{109}
0\leq y_{p_j}'(t)=G_{p_j}(t)^{1/(p_j-1)}
\leq \delta  ^{1/(p_j-1)},
\end{equation}
and
\[
0\leq y_{p_j}(t)-1
=\int_0^t y_{p_j}'(s)\,\din s
\leq {r_0}    \delta  ^{1/(p_j-1)}.
\]
This proves \eqref{eq:n-one-exponential-collapse} and
\eqref{eq:n-one-C1-limit}.

Local boundedness \eqref{114} of $u_{p_j}$, \eqref{eq:n-one-C1-limit}, and
$p_j-1\to0$ give
\begin{equation}\label{118}
(p_j-1)\log u_{p_j}\to   0,\quad
\beta _{p_j}\log u_{p_j}\to   0.
\end{equation}
The convergence is locally uniform in $[-r_0,r_0]$; hence
\eqref{eq:n-one-power-limits} follows.

\medskip

\noindent\textbf{Step 2.} From \eqref{118}, \eqref{eq:n-one-F-local-limit}, and
\eqref{eq:n-one-exponential-collapse}, we have
\[
F_j(t,y_{p_j}(t),y_{p_j}'(t))
\to    F_1(t,1,0)
\qquad\text{uniformly on }[0,{r_0}    ].
\]
This proves \eqref{eq:n-one-source-limit}.

Since
\[
G_{p_j}'=k_{p_j},
\qquad
G_{p_j}(0)=0,
\]
equation \eqref{eq:n-one-G-C1-limit} follows.

For every $t\in[0,{r_0}    ]$, \eqref{109} gives
\[
\int_0^tF_1(s,1,0)\,\din s=\lim_{j\to\infty}G_{p_j}(t)\leq \delta  <1.
\]
Thus \eqref{eq:n-one-subcritical-limit} holds.

\medskip
\noindent\textbf{Step 3.}
Equations \eqref{eq:n-one-zj} and
\eqref{eq:n-one-G-C1-limit} give \eqref{eq:n-one-z-limit}. Moreover,
\[
z_j'(x)=k_{p_j}(|x|)
\to    F_1(|x|,1,0)
\]
{uniformly on $[-r_0,r_0]$}, including at $x=0$. Hence
\eqref{eq:n-one-limit-flux-equation} holds.

Finally, on ${(-r_0,r_0)}$,
\[
u\equiv1,
\qquad
Du=0,
\qquad
(z,Du)=0=|Du|.
\]
Definition \ref{113} gives
\eqref{eq:n-one-limit-one-laplacian}. \end{proof}

%%
%%
%%
%%
%%%
%%
%%
%%%

%%
%%
%%
%%
%%%
%%
%%
%%%

%%%%%%%%%%%%%%%%%%%%%%%%%%%%%%%%%%%%%%%%%%%%%%%%%%%%%%%%%%%%%%%%%%%%%%%%%%%%%%%%%%%%%%%%%%%%%%%%%%%%%%%%%%%%%%%%%%%%%%%%%%%%%%%%%%%%%%%%%%%%%%%%%%%%%%%%%%%%%%%%%%%%%%%%%%%%%%%%%%%%%%%%%%%%%%%%%%%%%%%%%%%%%%%%%%%%%%%%%%%%%%%%%%%%%%%%%%%%%%%%%%%%%%%%%%%%%%%%%%%%%%%%%%%%%%%%%%%%%%%%%%%%%%%%%%%%%%%%%%%%%%%%%%%

\vspace{1cm}

%\noindent {\bf Author contributions:} All authors have contributed equally to this work for writing, review and editing. All authors have read and agreed to the published version of the manuscript.

\noindent {\bf Funding:} This work was supported by the Conselho Nacional de Desenvolvimento Científico e Tecnológico (CNPq), Grant No. 150680/2025-2.

\noindent {\bf Data Availability:} No data were used for the research described in the article.

\noindent {\bf Declarations}

\noindent {\bf Conflict of interest:} The author declares no conflict of interest.

%\noindent {\bf Declaration of generative AI and AI-assisted technologies in the writing process:} During the preparation of this work, the authors used the AI ChatGPT and the AI Gemini to correct grammar and orthography in the text. After using this tool/service, the authors reviewed and edited the content as needed and takes full responsibility for the content of the publication.

%%%%%%%%%%%%%%%%%%%%%%%%%%%%%%%%%%%%%%%%%%%%%%%%%%%%%%%%%%%%%%%%%%%%%%%%%%%%%%%%%%%%%%%%%%%%%%%%%%%%%%%%%%%%%%%%%%%%%%%%%%%%%%%%%%%%%%%%%%%%%%%%%%%%%%%%%%%%%%%%%%%%%%%%%%%%%%%%%%%%%%%%%%%%%%%%%%%%%%%%%%%%%%%%%%%%%%%%%%%%%%%%%%%%%%%%%%%%%%%%%%%%%%%%%%%%%%%%%%%%%%%%%%%%%%%%%%%%%%%%%%%%%%%%%%%%%%%%%%%%%%%%%%%%%%%%%%%%%%%%%%%%%%%%%%%%%%%%%%%%%%%%%%%%%%%%%%%%%%%%%%%%%%%%

%\bibliographystyle{plain}
 
 \bibliographystyle{abbrv}

   \bibliography{ref}

\end{document}